\documentclass[12pt,reqno]{amsart}

\usepackage{amsmath,amsfonts,amsthm,amssymb,color}
\usepackage[T1]{fontenc}
\usepackage[latin1]{inputenc}
\usepackage{xy} 
\usepackage{graphicx}

 \newcommand{\be}{\mathbf{E}}
 \newcommand{\bp}{\mathbf{P}}

\newcommand{\ott}{[0,T]}

\newcommand{\ca}{{\mathcal A}}

\newcommand{\cac}{{\mathcal C}}
\newcommand{\cd}{{\mathcal D}}

\newcommand{\cf}{{\mathcal F}}

\newcommand{\ct}{{\mathcal T}}
\newcommand{\cu}{{\mathcal U}}

\newcommand{\al}{\alpha}
\newcommand{\ga}{\gamma}
\newcommand{\gga}{\Gamma}
\newcommand{\ep}{\varepsilon}
\newcommand{\ka}{\kappa}

\newcommand{\si}{\sigma}

\newcommand{\lcl}{\left\{}
\newcommand{\rcl}{\right\}}
\newcommand{\lp}{\left(}
\newcommand{\rp}{\right)}
\newcommand{\lc}{\left[}
\newcommand{\rc}{\right]}

\newcommand{\bean}{\begin{eqnarray*}}
\newcommand{\eean}{\end{eqnarray*}}
\newcommand{\ben}{\begin{enumerate}}
\newcommand{\een}{\end{enumerate}}
\newcommand{\beq}{\begin{equation}}
\newcommand{\eeq}{\end{equation}}

\newtheorem{thm}{Theorem}[section]
\newtheorem{lem}[thm]{Lemma}

\newtheorem{prop}[thm]{Proposition}
\newtheorem{defn}[thm]{Definition}
\theoremstyle{remark}

\def\half{{\frac{1}{2}}}

\newcommand{\R}{\mathbb{R}}
\newcommand{\C}{\mathbb{C}}
\newcommand{\Z}{\mathbb{Z}}
\newcommand{\RR}{\mathbb{R}}

\newcommand{\EX}{\mathbf{E}}

\newcommand{\II}{{\rm i}}

\begin{document} 

\title[Discretizing the fractional L\'evy Area]{Discretizing the fractional L\'evy Area}

\author{A. Neuenkirch, S. Tindel and J. Unterberger}
\address{A. Neuenkirch, Johann Wolfgang Goethe Universit\"at,
Institut f\"ur Mathematik,
Robert Mayer Str. 10,
D-60325 Frankfurt am Main,
Germany} \email{neuenkirch@math.uni-frankfurt.de}

\address{Samy Tindel, Institut {\'E}lie Cartan Nancy, Universit\'e de Nancy 1, B.P. 239,
54506 Vand{\oe}uvre-l{\`e}s-Nancy Cedex, France}
\email{tindel@iecn.u-nancy.fr}

\address{J\'er\'emie Unterberger, Institut {\'E}lie Cartan Nancy, Universit\'e de Nancy 1, B.P. 239,
54506 Vand{\oe}uvre-l{\`e}s-Nancy Cedex, France}
\email{jeremie.unterberger@iecn.u-nancy.fr}

\thanks{Supported by the DAAD (PPP-Procope D/0707564) and Egide (PHC-Procope 17879TH)}

\subjclass{Primary 60H35; Secondary 60H07, 60H10, 65C30}
\date{\today}
\keywords{fractional Brownian motion, L\'evy area, approximation schemes}

\begin{abstract}
In this article, we give sharp bounds  for the Euler- and  trapezoidal discretization of the L\'evy area associated to a $d$-dimensional fractional Brownian motion.
We show that there are three different regimes for the exact root mean-square convergence rate of the Euler scheme. For $H<3/4$ the exact convergence rate is $n^{-2H+1/2}$, where $n$ denotes the number of the discretization subintervals, while for $H=3/4$ it is $n^{-1} ( \log(n))^{1/2}$ and for $H>3/4$ the exact rate is $n^{-1}$. Moreover, the trapezoidal scheme has exact convergence rate $n^{-2H+1/2}$ for $H>1/2$.
Finally, we also derive the asymptotic error distribution of the Euler scheme. For $H \leq 3/4$ one obtains a Gaussian limit, while for $H>3/4$ the limit distribution is of Rosenblatt type.
\end{abstract}

\maketitle

\section{Introduction and Main Results}

Let $B=(B^{(1)},\ldots,B^{(d)})$ be a $d$-dimensional fractional Brownian motion (fBm) with Hurst parameter $H \in  (1/4,1)$ indexed by $\RR$,  i.e.~$B$ is  composed of $d$ independent centered Gaussian processes whose covariance function is given by
$$
R_H(s,t) = \frac{1}{2} \lp |s|^{2H} + |t|^{2H} - |t-s|^{2H} \rp, \qquad s,t \in\RR.
$$
For an arbitrary $T>0$, a typical differential equation on $\ott$ driven by $B$ can be written as
\beq\label{eq:1}
Y_t= a + \int_0^t \si(Y_s) \, dB_s, \quad t\in \ott,
\eeq
 where  $a \in \R^n$ is a given initial condition and $\si: \R^n\to\R^{n,d}$ is  sufficiently smooth.
During the last  years, the rough paths theory has allowed to handle several aspects of differential equations like (\ref{eq:1}), ranging from existence and uniqueness results (see \cite{FV,LQ} for equations of type  (\ref{eq:1}) and \cite{CF,GT,NNT} for extensions to other kind of systems) to density estimates \cite{CFV} or ergodic theorems \cite{Ha}. 

\smallskip

It is also important, and in fact at the very core of the rough path analysis, to  derive good numerical approximations for fractional differential equations like~(\ref{eq:1}). This problem has so far been considered in three type of situations: \textit{(i)} When $H>1/2$, it is proved independently in \cite{Da} and \cite{MS} that the Euler scheme associated to equation (\ref{eq:1}), based on a grid $\{iT/n;\, i\le n\}$, converges with the rate  $n^{-(2H-1)+ \varepsilon}$ for arbitrarily small $\varepsilon > 0$.  The exact rate of convergence of the Euler scheme is  computed in \cite{NN} in the particular case of a one-dimensional equation. \textit{(ii)} In the Brownian case $H=1/2$, there exists a huge amount of literature on approximation schemes for SDEs, and we just send the interested reader to the references \cite{KP, MT} for a complete overview of the topic. \textit{(iii)} For $1/3<H<1/2$, the rough path strategy in order to solve  equation (\ref{eq:1}), see e.g.  \cite{FV,Gu,LQ},  tells us that one should use at least a Milstein-type scheme 
in order to approximate its solution.  Moreover, it can be easily seen that for $H<1/2$ the standard Euler scheme does not converge and in fact explodes for stepsizes going to zero, even in the one-dimensional case. Indeed, consider for instance the one-dimensional SDE
$$ dX_t=X_t \, dB_t, \qquad X_0=1,$$
whose exact solution is given by $X_t=\exp(B_t)$. The Euler approximation for this equation at $t =1$ is given by
$$  X^{(n)}_1 = \prod_{k=0}^{n-1} (1+ (B_{(k+1)/n}-B_{k/n})).$$
So for $n \in \mathbb{N}$ sufficiently large and using a Taylor expansion, we have
\begin{align*}
 X_1 -   X^{(n)}_1 & = \exp(B_1) - \exp\Big( \sum_{k=0}^{n-1} \log  (1+ (B_{(k+1)/n}-B_{k/n})) \Big) \\
                              &= \exp(B_1) - \exp\Big (B_1 -\frac12 \sum_{k=0}^{n-1}|B_{(k+1)/n}-B_{k/n}|^2 + \rho_n \Big),
\end{align*}
where  $\rho_n \stackrel{{\rm Prob.}}{\longrightarrow} 0 $ for $n \rightarrow \infty$ for $H>1/3$. Now it is well known that 
$$ \sum_{k=0}^{n-1}|B_{(k+1)/n}-B_{k/n}|^2 \stackrel{{\rm a.s.}}{\longrightarrow} \infty$$
for $H<1/2$, so we have  $ X^{(n)}_1 \stackrel{{\rm Prob.}}{\longrightarrow} \infty $.
However, Milstein-type schemes are known to be convergent for such a one-dimensional equation, see \cite{GN}.

\smallskip

For the general multi-dimensional equations  of type (\ref{eq:1}), a Milstein-type scheme   is studied in  \cite{Da}: set $\overline{Y}_0 = a$, and for a grid given by $t_k=kT/n$, $k=0, \ldots, n-1$, let
\begin{align} \label{eq:misltein}
\overline{Y}_{t_{k+1}} =  & \overline{Y}_{t_k} +  \sum_{i=1}^n \sigma^{(i)} ( \overline{Y}_{t_{k}}) (B_{t_{k+1}} - B_{t_{k}})  \\ \qquad  &+ \sum_{i,j=1}^n \mathcal{D}^{(i)}\sigma^{(j)}  (\overline{Y}_{t_{k}}) \int_{t_k}^{t_{k+1}} (B^{(i)}_s -  B^{(i)}_{t_k}) \, d B_s^{(j)} \nonumber ,
\end{align}
for $k=0, \ldots, n-1$,
where $\mathcal{D}^{(i)}$ is the differential operator  $\sum_{l=1}^{d} \sigma_l^{(i)} \partial_{x_l}$. Davie then proves that  this scheme has convergence rate  $n^{-(3H-1) + \varepsilon }$, and this result has been extended in \cite{FV} in an abstract setting, to higher order schemes for a rough path with a given regularity.

\smallskip

The above Milstein-type scheme (\ref{eq:misltein}) requires  knowledge of the iterated integrals
\begin{align} X_t^{(i,j)}= \int_{0}^t  B^{(i)}_s\, d B_s^{(j)} , \qquad t \in [0,T], \quad i,j=1, \ldots n, \label{fla-1} \end{align}
whose explicit distribution is unknown for $i \neq j$. Thus discretization procedures for  (\ref{fla-1}) are  crucial for an implementation of this numerical method. This has already been addressed in \cite{CQ}, where dyadic linear approximations of the fBm $B$ are used in order to define a  Wong-Zakai-type approximation $\widehat{X}^{n}$ of $X$. In the last reference, the process $\widehat X^{n}$ is shown to converge almost surely in $p$-variation distance, and the  (non-optimal) error bound $$\be |\widehat{X}^{n}_T-X_T|^2 \leq C  \cdot 2^{-n(4H-1)/2}$$ is also determined. 
The current article takes up this kind of program, and we consider the approximation of 
\begin{align} \label{fla} X_t = \int_{0}^t B_s^{(1)} \, d B_s^{(2)}, \qquad t \in [0,T]  \end{align} 
by the Euler- and a {  trapezoidal } scheme based on equidistant discretizations.

 For the approximation of (\ref{fla}) the standard Euler method has the explicit expression
\beq\label{euler}
X^{n}_T= \sum_{i=0}^{n-1} \ B_{iT/n}^{(1)} \lp B_{(i+1)T/n}^{(2)}  -   B_{iT/n}^{(2)} \rp.
\eeq
The results we obtain for the Euler scheme are then of two kinds. First we determine the exact $L^2$-convergence rate.
\smallskip

\begin{thm}\label{thm:1.1} 
Let $X_T$  defined by (\ref{fla}) and its Euler approximation $X_T^n$ given by expression (\ref{euler}). Define the constants $\alpha_j(H)$, for $j=1,2,3$ by
\begin{multline*}
\alpha_1(H) =   \frac{H}{2}   \left( \beta(2H, 2H) +     \frac{1}{4H-1}  \right)
 + \frac{1}{2}  \left((1 - 2^{2H})  + \frac{2H-1}{4H-1}   + \frac{H 2^{4H}} {4H-1}  \right)  \\
 + H  \int_{0}^{1}  (  y^{2H}|  1+y|^{2H-1}   - y^{2H-1}|1+y|^{2H} ) \, d y
\end{multline*}
and
$$
\alpha_2(H) =   \alpha_1(H) + \frac{H^2(2H-1)^2}{2}  \zeta(4-4H), \qquad
\alpha_3(H) =  \frac{1}{4} \frac{H^2 (2H-1)}{4H-3}.
$$
Then we have
\begin{align*}
\EX |X_T -{X}_T^{n}   |^2 = \left \{ \begin{array}{rllcc}
\alpha_1(H)   \cdot T^{4H} \cdot n^{-4H+1} &+ &o( n^{-4H+1} )   &\textrm{for}   &  H \in (1/4, 1/2),  \\
\alpha_2(H) \cdot T^{4H} \cdot n^{-4H+1} &+& o( n^{-4H+1} ) & \textrm{for}   &  H \in (1/2, 3/4), \\
\frac{9}{128} \cdot T^{4H}\cdot  \log(n) n^{-2} &+ &o(\log(n) n^{-2}  ) & \textrm{for}   &  H = 3/4, \\
 \alpha_3(H) \cdot T^{4H}\cdot n^{-2} &+ &o( n^{-2} ) & \textrm{for}   & H \in (3/4,1).
\end{array} \right.
\end{align*}
\end{thm}

\smallskip

Observe that for the case $H=1/2$, i.e. for the approximation of the Wiener L\'evy area,   one obtains by straightforward computations that
$  \EX |X_T -{X}_T^{n}   |^2 = \frac{T^2}{2} \cdot n^{-1},$ which is compatible with our  Theorem \ref{thm:1.1}, since
$$      \lim_{H \rightarrow 1/2, \, H<1/2} \alpha_1(H) =  \lim_{H \rightarrow 1/2, \, H>1/2} \alpha_2(H)  = \frac{1}{2}. $$

 The convergence rate breaks up into several regimes which are reminiscent of the cases obtained in \cite{NoNuTu, TV} concerning { weighted quadratic variations} of the one-dimensional fBm. In particular, the convergence rate  does not improve for $H \geq 3/4$, i.e. is equal to $n^{-1}$ independently of $H$. Finally, note that our study starts obviously at $H=1/4^+$, since the L\'evy area is not even defined for $H\le 1/4$.

\smallskip
Using a trapezoidal rule for the approximation of the integral leads to the following  scheme, which coincides with the Wong-Zakai approximation used in \cite{CQ}:
\beq \widehat{X}_T^n= \frac{1}{2}  \sum_{i=0}^{n-1} \left( B_{iT/n}^{(1)} + B_{(i+1)T/n}^{(1)} \right) \left( B_{(i+1)T/n}^{(2)}  -   B_{iT/n}^{(2)} \right). \eeq
This { trapezoidal} scheme avoids the "breakdown" of the convergence rate of the Euler scheme for $H \geq 3/4$.

\smallskip
 \begin{thm}\label{milstein} Let $H>1/2$.
Then we have
$$          \EX |X_T- \widehat{X}_T^n|^2 = \alpha_4(H) \cdot  T^{4H} \cdot n^{-4H+1}+ o(n^{-4H+1}) ,         $$
where
$$   \alpha_4(H)=  \  \EX  \int_1^2  \Big( B_{s_1} ^{(1)} - \frac{1}{2}(B_{1}^{(1)}-B_{2}^{(1)}) \Big) \, d B_{s_1} ^{(2)}  \int_0^3 \Big( B_{s_2} ^{(1)}- \frac{1}{2}(B_{1}^{(1)}-B_{2}^{(1)})  \big) \, d B_{s_2} ^{(2)}.  $$
\end{thm}

\smallskip

Note that the constant $\alpha_4(H)$ could also be expressed in terms similar to $\alpha_1(H)$. However, we think that this gives no further insight and thus we omit it here.
We strongly suspect that the root mean square convergence rate $n^{-2H+1/2}$, which is obtained 
by this { trapezoidal} scheme, is the best possible. In other words, we conjecture that the conditional expectation of $X_T$ given $ B_{T/n} , B_{2T/n}, \ldots B_{T}$ satisfies
$$ \EX \big | X_T - \EX ( X_T \, | \, B_{T/n} , B_{2T/n}, \ldots B_{T} ) \big|^2 = C(H) \cdot T^{4H} \cdot n^{-4H+1}+ o(n^{-4H+1}) , $$ where  $C(H)>0$.

\smallskip

The third result in  this article is a refinement of Theorem \ref{thm:1.1}, meaning that we obtain a limit theorem for the asymptotic error distribution of the Euler scheme.

\smallskip

 \begin{thm}\label{thm:1.2} 
Let $X_T$, $X_T^n$ and $\alpha_1(H), \alpha_2(H), \alpha_3(H)$ defined as above. Moreover, let $Z$ be  a standard normal random variable.
Then :

\begin{enumerate}
\item Case $1/4<H\le 3/4$: the following central limit theorems hold:
\begin{align*}
\lim_{n\to\infty}\,  n^{2H-1/2}\, (X_T -{X}_T^{n})  \,  \stackrel{(d)}{=}\,  \left \{ \begin{array}{rlcc}
\sqrt{\alpha_1(H)}T^{2H} \cdot Z   &\textrm{for}   &  H \in (1/4, 1/2),  \\
\sqrt{\alpha_2(H)}T^{2H} \cdot Z & \textrm{for}   &  H \in (1/2, 3/4) 
\end{array} \right.
\end{align*}
and 
$$  \lim_{n\to\infty} \, n (\log(n))^{-1/2} \,  ( X_T- X_T^n )  \,   \stackrel{(d)}{=}  \,  \frac{3}{4 \sqrt{8}} T^{2H} \cdot Z $$ for $H=3/4$.

\item Case $H>3/4$:  let  $R_1$ and $R_2$ be two independent Rosenblatt processes  (see
Section~5 for a definition). 
Then it holds
 $$  \lim_{n \to \infty}\,  n \, (  X_T - X_T^n   ) \,  \stackrel{(d)}{=} \, \sqrt{2 \alpha_4(H)} T^{2H}  \cdot (R_1-R_2).$$

\end{enumerate}
\end{thm}

\smallskip

 Let us say a few words about the methodology we have adopted in order to prove Theorem~\ref{thm:1.2}. It should be mentioned first that we have used the analytic approximations introduced in \cite{Unt08a} in order to define the Lévy area $X$, which allows to use some elegant complex analysis methods for moments estimates in this context. Then, for $H \in (1/4,3/4)$, the central limit type results are obtained through the criterion introduced in \cite{NP} for random variables in a fixed chaos.  { For this we control  the  fourth moments of $X$ with the help of
(Feynman) diagrams.}  For the case $H\ge 3/4$ we proceed in a different way. Here the Milstein approximation of $X_T$ performs better than the Euler method. Then expressing the differences between both schemes as the sum of  quadratic variations  for two independent one-dimensional fBms, thanks to a simple geometrical trick given in \cite{No}, one obtains the limit theorems for $H \geq 3/4$ using the results of \cite{TV}.  In particular, this leads to the  Rosenblatt type limit distribution as in \cite{TV}.  For the  trapezoidal scheme, whose error behaves like the second  order quadratic variations of fBm, see e.g. \cite{Be}, a central limit theorem could be also derived using the criterion in \cite{NP}, but we omit this here  for the sake of conciseness.

\smallskip

The remainder of this article is structured as follows.  Integrals with respect to the fractional Brownian motion will always be understood as limits of analytic integrals as in \cite{Unt08a}. We thus recall the definition of the analytic fBm, as well as some preliminaries at Section \ref{sec:prelim}.   Section \ref{sec:mean-square} contains the proofs of Theorem {\ref{thm:1.1}} and  {\ref{milstein}}. The proof  of Theorem {\ref{thm:1.2}}
is given in  Sections \ref{sec:fourth-moments} and \ref{sec:Hgeq34}.

\section{Definition of the analytic fbm and preliminaries}\label{sec:prelim}

This section is devoted to recall the definition of the fractional Brownian motion introduced in \cite{Unt08a}, and to state some of the properties of this process which will be used in the sequel. All the random variables introduced here will be defined on a complete probability space $(\cu,\cf,\bp)$, without any further mention (notice the unusual notation $\cu$ for our probability space, due to the fact that the letter $\Omega$ will serve for the complex domains we consider in the sequel). The following kernels will also be essential for our future computations:
\begin{defn}[$\eta$-regularized power functions]\label{def:K'}
For $\beta\in\R\setminus\Z$ and $\eta>0$ let
\begin{equation*} 
[x]_{\eta}^{\pm,\beta}=(\pm\II x+\eta)^{\beta} 
\quad\mbox{and}\quad
[x]_{\eta}^{\beta}=2\Re [x]_{\eta}^{\pm,\beta}=[x]_{\eta}^{+,\beta}+[x]_{\eta}^{-,\beta}. 
\end{equation*}
Then, for $\eta>0$ and $x,y\in\R$, define $K^{',\pm}(\eta;x,y)$ as
$$
K^{',\pm}(\eta;x,y)=\frac{H(1-2H)}{2\cos\pi H} (\pm\II(x-y)+\eta)^{2H-2}
=\frac{H(1-2H)}{2\cos\pi H} [x-y]_{\eta}^{\pm,2H-2}.
$$
Set also
$$
K'(\eta;x,y):=2\Re K^{',\pm}(\eta;x,y)=K^{',+}(\eta;x,y)+K^{',-}(\eta;x,y).
$$
Notice that the above kernels  are well-defined on our prescribed domain $\R_+^*\times\R\times\R$.
\end{defn}

\subsection{Definition of the analytic fBm}
The article \cite{Unt08a} introduces the fractional Brownian motion as the real part of 
the trace on $\R$ of an analytic process $\gga$ (called: {\em analytic fractional
Brownian motion} \cite{TinUnt08}) defined on the complex upper-half plane $\Pi^+=\{z\in\C;\, \Im(z)>0\}$. This is achieved by first noticing that the kernel $K'(\eta)$ is  positive definite and represents (for every fixed $\eta>0$)
 the covariance of 
of a real-analytic centered Gaussian process with real time-parameter $t$. The easiest way to see it is to make use of the
following explicit series expansion: for $k\ge 0$ and $z\in\Pi^+$, set
\begin{equation} 
f_k(z)=2^{H-1} \sqrt{\frac{H(1-2H)}{2\cos\pi H}}
 \sqrt\frac{\mathbf{\Gamma}(2-2H+k)}{\mathbf{\Gamma}(2-2H) k!}
\left( \frac{z+\II}{2\II} \right)^{2H-2} \left( \frac{z-\II}{z+\II}\right)^k,  
\end{equation}
where $\mathbf{\Gamma}$ stands for the usual Gamma function. Then these functions are well-defined on $\Pi^+$, and it can be checked that one has
\begin{equation*} 
\sum_{k\ge 0} f_k\lp x+\II \frac{\eta_1}{2}\rp \overline{f_k\lp y+\II\frac{\eta_2}{2}\rp} = K^{',-}\lp\half\lp\eta_1+\eta_2\rp;x,y\rp.
\end{equation*} 
Define more generally  a Gaussian process with time parameter $z\in \Pi^+$ as follows:
\begin{equation} \label{eq:7}
\Gamma'(z)=\sum_{k\ge 0} f_k(z)\xi_k 
\end{equation}
 where $(\xi_k)_{k\ge 0}$ are independent standard complex
Gaussian variables, i.e. $\EX[\xi_j \xi_k]=0$, $\EX[\xi_j \bar{\xi}_k]=\delta_{j,k}$.
The Cayley transform $z\mapsto  \frac{z-\II}{z+\II}$ maps $\Pi^+$ to $\cd$, where $\cd$ stands for the unit disk of the complex plane. This allows to prove trivially that the series defining $\Gamma'$ is a random entire series which may
be shown to be analytic on the unit disk. Hence the process $\Gamma'$ is analytic on $\Pi^+$. Furthermore
note that, restricting to the horizontal line $\R+\II\frac{\eta}{2}$, the following identity holds true:
$$ 
\EX[\Gamma'(x+\II\eta/2)\overline{\Gamma'(y+\II\eta/2)}]=K^{',-}(\eta;x,y).
$$
 
\smallskip

One may now integrate the process $\Gamma'$ over any path $\gamma:(0,1)\to\Pi^+$ with endpoints $\gamma(0)=0$
and $\gamma(1)=z\in\Pi^+\cup\R$ (the result does not depend on the particular path but only on the endpoint $z$).
The result is a process $\Gamma$ which is still analytic on $\Pi^+$. Furthermore,
one may retrieve the fractional Brownian motion by considering the real part of the boundary value of $\Gamma$
on $\R$. Another way to look at it is to define $\Gamma_t(\eta):=\Gamma(t+\II\eta)$
 as a regular process living on $\R$, and to remark that the real part of $\Gamma(\eta)$ converges
when $\eta\to 0$ to fBm.  In the following Proposition, we give precise statements which
summarize what has been said up to now:

\begin{prop}[see \cite{Unt08a,TinUnt08}]
Let $\gga'$ be the process defined on $\Pi^+$ by relation (\ref{eq:7}).
\begin{enumerate}
\item
Let $\gamma:(0,1)\to\Pi^+$ be a continuous path with endpoints $\gamma(0)=0$ and $\gamma(1)=z$,
 and set $\Gamma_z=\int_{\gamma}\Gamma'_u \, du$. Then $\Gamma$ is an analytic process on $\Pi^+$.
 Furthermore, as $z$ runs along any path in $\Pi^+$ going to $t\in\R$, the random variables $\Gamma_z$
 converge almost surely to a random variable called again $\Gamma_t$.

\item

The family $\{\Gamma_t;\, t\in\R\}$ defines a  centered Gaussian complex-valued process whose
paths are almost surely $\kappa$-H\"older  continuous for any $\kappa<H$. Its real
part $B_t:=2\Re\Gamma_t$  has the same law as fBm.

\item
The family of centered Gaussian real-valued processes $B_t(\eta):=2\Re \Gamma_{t+\II\eta}$ converges a.s. to $B_t$ in $\alpha$-H\"older norm for any $\al<H$, on any interval of the form $[0,T]$ for an arbitrary constant $T>0$. Its infinitesimal
covariance kernel $\EX B'_x(\eta) B'_y(\eta)$ is $K'(\eta;x,y)$.

\end{enumerate}
\label{prop:Gamma-process}
\end{prop}

\subsection{Definition of the L\'evy area}
Let us describe a natural possible definition of the Lévy area associated to $\gga$. Since the process $B_t(\eta):=2\Re \Gamma_{t+\II\eta}$ is a smooth one, one can define the following integral in the Riemann sense for all $0\le s<t$ and $\eta>0$:
\begin{equation}\label{eq:8a}
\ca_{st}(\eta)=\int_{s}^{t} dB_{u_1}^{(2)}(\eta) 
\int_{s}^{u_1} dB_{u_2}^{(1)}(\eta).
\end{equation}
It turns out that $\ca(\eta)$ converges in some Hölder spaces, in a sense which can be specified as follows. Let $T$ be an arbitrary positive constant, $\cac_j$ be the set of continuous complex-valued functions defined on $[0,T]^j$, and for $\mu>0$, define a space $\cac_2^\mu$ of $\mu$-Hölder functions on $[0,T]^2$ by
\begin{equation}\label{eq:9}
\|f\|_{\mu} :=
\sup_{s,t\in\ott}\frac{|f_{ts}|}{|t-s|^\mu}
\quad\mbox{and}\quad
\cac_2^\mu(V)=\lcl f \in \cac_2(\Omega;V);\, \|f\|_{\mu}<\infty  \rcl.
\end{equation}
The $\mu$-Hölder semi-norm for a function $g\in\cac_1$ is then defined by setting $h_{st}=g_t-g_s$ as an element of $\cac_2$, and $\|g\|_{\mu}:=\|h\|_{\mu}$ in the sense given by (\ref{eq:9}).

\smallskip

According to \cite{Unt08a,TinUnt08}, the L\'evy area $\ca$ of $B$ can then be defined in the following way:
\begin{prop}\label{prop:2.3}
Let $T>0$ be an arbitrary constant, and for $s,t\in\ott^2$, $\eta>0$, define $\ca_{st}(\eta)$ as in equation (\ref{eq:8a}). Consider also $0<\ga<H$. Then:
 
\smallskip

\noindent
{\bf (1)}
For any $p\ge 1$, the couple $(B(\eta),\ca(\eta))$ converges in $L^p(\Omega;\cac_1^{\ga}(\ott;\R)\times\cac_2^{2\ga}(\ott^2;\R))$ to a couple $(B,\ca)$, where $B$ is a fractional Brownian motion.
 
\smallskip

\noindent
{\bf (2)}
The increment $\ca$ satisfies the following algebraic relation: 
$$
\ca_{st}-\ca_{su}-\ca_{ut}=\big( B_t^{(2)}- B_u^{(2)}\big) \,
\lp B_u^{(1)}- B_s^{(1)}\rp,
$$
for $s,u,t\in\ott$.
\end{prop}
Notice that the algebraic property (2) in Proposition \ref{prop:2.3} is the one which qualifies $\ca$ to be a reasonable definition of the L\'evy area of $B$.

\smallskip

It  will be essential for us to estimate the moments of $\ca$. For this  we will use the following definition:
\begin{defn}\label{def:K}
For $\eta>0$ and $a_1,a_2\in\R$, let us define the function $K_{a_1,a_2}(\eta; \cdot,\cdot)$ on $
\R\times\R$ by
\beq\label{eq:10a}
K_{a_1,a_2}(\eta;x_1,x_2)=\int_{a_1}^{x_1} dy_1 \int_{a_2}^{x_2} dy_2 K'(\eta;y_1,y_2).
\eeq
Notice then that, invoking the conventions of Definition \ref{def:K'}, we have
\begin{multline}\label{eq:10}
K_{a_1,a_2}(\eta;x_1,x_2)=\frac{1}{4\cos(\pi H)}
\big( [x_1-x_2]_{\eta}^{2H}-[x_1-a_2]_{\eta}^{2H}-[a_1-x_2]_{\eta}^{2H}  \\
+[a_1-a_2]_{\eta}^{2H} \big).
\end{multline}
\end{defn}
We also state the classical Wick lemma for further use.
\begin{prop}
Let $Z=(Z_1,\ldots,Z_{2N})$ be a centered Gaussian vector. Then
\begin{equation} \EX[Z_1\cdots Z_{2N}]=\sum_{(i_1,i_2),\ldots,(i_{2N-1},i_{2N})}
 \prod_{j=1}^N \EX[Z_{i_{2j}} Z_{i_{2j+1}}] 
\end{equation}
where the sum ranges over the $(2N-1)!!=1 \cdot 3 \cdot 5\cdots (2N-1)$ couplings of the indices
$1,\ldots,2N$.
\label{prop:1:Wick}
\end{prop}
We can now give the announced expression for the moments of $\ca(\eta)$ (recall that $\ca(\eta)$ is defined by (\ref{eq:8a})):
\begin{lem}
Let $N\ge 1$ and $\{s_i,t_i;\, i\le 2N\}$ be a family of real numbers satisfying $s_i<t_i$. Then
\begin{align} \label{eq:11}
&\EX\left[\prod_{j=1}^{2N}
{\mathcal A}_{s_j,t_j}(\eta) \right]=  
\int_{s_1}^{t_1} dx_1 \cdots \int_{s_{2N}}^{t_{2N}} dx_{2N} \\
&   \sum_{(i_1,i_2),\ldots,(i_{2N-1},i_{2N})}
\sum_{(j_1,j_2),\ldots,(j_{2N-1},j_{2N})}  
\prod_{k=1}^N K'(\eta;x_{i_{2k-1}},x_{i_{2k}}) \ .\
\prod_{k=1}^N K_{s_{j_{2k-1}},s_{j_{2k}}}(\eta;x_{j_{2k-1}},x_{j_{2k}}). \nonumber
\end{align}
\label{lemma:1:Area2N}
\end{lem}

\begin{proof}
By definition of the approximation $\ca(\eta)$, we have
\beq\label{eq:13}
\EX\left[\prod_{j=1}^{2N}
{\mathcal A}_{s_j,t_j}(\eta) \right]
 = \prod_{j=1}^{2N} \int_{s_j}^{t_j} dx_j \int_{s_j}^{x_j} dy_j 
 \EX\left[B_{x_1}^{'(1)}(\eta) B_{y_1}^{'(2)}(\eta)
\cdots B_{x_{2N}}^{'(1)}(\eta) B_{y_{2N}}^{'(2)}(\eta) \right] 
\eeq
Our claim stems then from a direct application of  Proposition \ref{prop:Gamma-process} point (3), Proposition~\ref{prop:1:Wick} and Definition \ref{def:K}. 

\end{proof}

\subsection{Analytic preliminaries}
We gather here some elementary integral estimates which turn out to be essential  for our computations. The first one concerns the behavior of the kernel $K_{a_1,a_2}$ given at Definition \ref{def:K'} when $|a_1-x_1|,|a_2-x_2|$ are of order $1$ and $|x_1-x_2|$ is large.
\begin{lem} \label{lem:K}
Assume $\eta,|x_1-a_1|,|x_2-a_2|\le 1$ and
$|x_1-x_2|,|a_1-x_2|,|a_2-x_1|,|a_1-a_2|$ are bounded from below by a positive constant $C$. Then
$$
|K_{a_1,a_2}(\eta;x_1,x_2)|\le C
\left( \min(|x_1-x_2|,|a_1-x_2|,|a_2-x_1|,|a_1-a_2|)\right)^{2H-2}.
$$
\end{lem}

\begin{proof}
The proof is elementary using the integral expression (\ref{eq:10}) for $K_{a_1,a_2}$.

\end{proof}

\smallskip

We shall also need to estimate convolution integrals of the form
$\int_0^t K'(\eta;z,u) f(u)\ du$ or $\int_0^t K_{a,b}(\eta;z,u) f(u)\ du$.
 The following lemma gives a precise answer when $f$ is analytic on a neighborhood of $(a,b)$ for two given constants $a,b\in\R$, and multivalued with a power behavior near $a$ and $b$.

\begin{lem} (see \cite{Unt08b})
Fix two real constants $a,b$ with $a<b$, and let $f$ be a function in $L^1([a,b],\C)$. Define another function $\phi$ by $\phi:z\mapsto \int_a^b (-\II(z-u))^{\beta}
(u-a)^{\gamma}  f(u)\ du$ with $\gamma>-1$ and $\beta+\gamma\in\R\setminus\Z$. Then:
\begin{enumerate}
\item Assume $f$ is analytic in a (complex) neighborhood  of $s\in (a,b)$. Then $\phi$ has
an analytic extension to a complex neighborhood of $s$.

\item Assume $f$ is analytic in a complex neighborhood  of $a$. Then $\phi$ may be written
on a small enough neighborhood of $a$ as the multivalued function
\begin{equation} \phi(z)=(z-a)^{\beta+\gamma+1} F(z)+G(z) \end{equation}
where both $F$ and $G$ are analytic.

\item More precisely, the following continuity property holds: let $\Omega$
be a complex neighborhood of $[a,b]$ and
$\ep\in (0,1/2)$. If $f$
is analytic on a relatively compact domain $\tilde{\Omega}$ containing the closure $\bar{\Omega}$
of $\Omega$, then $\phi$ extends analytically to the cut domain
$\Omega_{{\rm cut}}:=\Omega\setminus((a+\R_-)\cup (b+\R_+))$ and writes
$(z-a)^{\beta+\gamma+1} F(z)+G(z)$ on $B(a,\ep (b-a))$ ($F,G$ analytic)
with
\begin{equation}
\sup_{\Omega_{{\rm cut}}\setminus(B(0,\ep (b-a))\cup B(b,\ep (b-a)))}
    |\phi|, \sup_{B(a,\ep (b-a))} |F|, \sup_{B(a,\ep (b-a))} |G|
\le C \sup_{\tilde{\Omega}} |f|
\end{equation}
for some constant $C$ which does not depend on $f$.
\end{enumerate}

\label{lem:exponents}

\end{lem}

\begin{proof}
Points (1) and (2) follow directly from \cite{Unt08b}, Lemmas 3.2 and 3.3. Point (3)
may be shown very easily by following the proof of the above two lemmas step
by step and using the analyticity of $f$. Note that (under the hypotheses
of (3)) $\phi$ is analytic on the larger domain
 $\tilde{\Omega}\setminus ((a+\R_-)\cup
(b+\R_+))$, but the method of contour deformation used in the proof
gives a bound for $\phi(z)$ which goes to infinity when $z$ comes closer
and closer to the boundary of $\tilde{\Omega}$ (hence the need for the
relatively compact inclusion of $\Omega$ into $\tilde{\Omega}$).

\end{proof}

We shall also need the following elementary lemma. Here and later on, we will write $x \lesssim y$ for $x,y \in \mathbb{R}$, if there exists a constant $C>0$ such that $x \leq C \cdot y$.

\begin{lem}

{\it Let $\alpha,\beta>-1$ and $0<a<b<1$. Then:
\beq \int_0^1 |t-a|^{\alpha} |t-b|^{\beta}\ dt\lesssim 1+|a-b|^{\alpha+\beta+1}. \eeq
}
\label{lem:2:alphabetaab}
\end{lem}

\begin{proof}

Let $\sigma_a(b)=\max(0,2a-b)$ and $\sigma_b(a)=\min(1,2b-a)$. Split the above integral into
$\int_0^{\sigma_a(b)}+\int_{\sigma_a(b)}^a+\int_a^b+\int_b^{\sigma_b(a)}+\int_{\sigma_b(a)}^1$.
We show that the integral over each subinterval is $\lesssim 1+|a-b|^{\alpha+\beta+1}$ (by
symmetry, it is sufficient to check this for the three first subintervals only). Now, a simple study of the function $t\mapsto |t-a|/|t-b|$ shows that
$c<\frac{|t-a|}{|t-b|}<C$ on $[0,\sigma_a(b)]$, so 
\beq \int_0^{\sigma_a(b)}  |t-a|^{\alpha} |t-b|^{\beta}\ dt\lesssim \int_0^{\sigma_a(b)}
(t-b)^{\alpha+\beta}\ dt\lesssim (b-a)^{\alpha+\beta+1}+b^{\alpha+\beta+1}. \eeq
If $\alpha+\beta+1<0$, resp. $\alpha+\beta+1>0$, then this is $\lesssim (b-a)^{\alpha+\beta+1}$,
resp. $\lesssim 1$. On $[\sigma_a(b),a]$, one has $c<\frac{|t-b|}{b-a}<C$ this time, so
\beq \int_{\sigma_a(b)}^a |t-a|^{\alpha} |t-b|^{\beta}\ dt\lesssim (b-a)^{\alpha+\beta+1}.\eeq
Finally, $\int_a^b |t-a|^{\alpha} |t-b|^{\beta}\ dt=\frac{{\bf \Gamma}(\alpha+1){\bf\Gamma}(\beta+1)}{{\bf\Gamma}(\alpha+\beta+2)} (b-a)^{\alpha+\beta+1}$, where $\bf \Gamma$ is the Gamma function.

\end{proof}



\section{Mean square error computations} \label{sec:mean-square}
 This section is devoted to prove Theorem \ref{thm:1.1} and Theorem \ref{milstein}. 
We will start with the error of the Euler scheme, the Milstein-type scheme will be considered later on.
 Note that we can decompose the error of the Euler scheme as 
$X_T - X_T^n=\sum_{i=1}^n  J_{i}^{n}$, where the random variables $J_{i}^{n}$ are defined by
$$ 
J_{i}^{n}= \int_{iT/n}^{(i+1)T/n} (B_s^{(1)} - B_{i/n}^{(1)}) \, d B_s^{(2)}
=\ca_{(iT)/n,(i+1)T/n}, \qquad i=0, \ldots, n-1,
$$
where $\ca_{st}$ is obtained as the $L^2$-limit of  $\ca_{st}(\eta)$ according to Proposition \ref{prop:2.3}. In particular, $\be[|X_T - X_T^n|^2]=\sum_{i,j}\be[J_i^n J_j^n]$, which means that we are first reduced to study the quantities $\be[J_i^n J_j^n]$ in terms of $i,j$ and $n$. Towards this aim, one can first remark that, since  fBm is self-similar and has stationary increments, we have
\beq\label{eq:15}
\EX  [J_{i}^{n}  J_j^{n} ]= T^{4H} n^{-4H} \EX  [ I_{i} I_j],
\eeq
with
$$ I_{i}= \int_{i}^{i+1} (B_s^{(1)} - B_{i}^{(1)}) \, d B_s^{(2)}
=\ca_{i,i+1}, \qquad i=0, \ldots, n-1, $$
We now show how to handle those terms.

\subsection{Some moment estimates}
The preliminary results we need in order to prove Theorem \ref{thm:1.1} are summarized in the following lemma:
\begin{lem}\label{lem:3.1}
Let  $\ca_{01}=\int_{0}^{1} B_s^{(1)} \, d B_s^{(2)}$ and $\ca_{12}=\int_{1}^{2} (B_s^{(1)} - B_1^{(1)}) \, d B_s^{(2)}$ be the double iterated integrals with respect to $B$ obtained by applying Proposition \ref{prop:2.3}. Define 
$$
c_1= \EX  \lc\left|  \ca_{01} \right|^{2}\rc,
\quad\mbox{and}\quad
c_2=\EX  \lc \ca_{01} \ca_{12} \rc.
$$
Then we have 
\beq\label{c1}
c_1= \frac{H}{2}\lp \beta(2H, 2H) + \frac{1}{4H-1} \rp,
\eeq
and
\begin{multline}\label{c2}
c_2 =    \frac{1}{4} (1 - 2^{2H})  + \frac{2H-1}{4(4H-1)}   + \frac{H 2^{4H} }{4(4H-1)}  \\
+ \frac{ H}{2}   \int_{0}^{1}  (  y^{2H}|  1+y|^{2H-1}   - y^{2H-1}|y+1|^{2H} ) \, d y.
\end{multline}
 \end{lem}
 
 \begin{proof}
 Both identities are obtained thanks to the same kind of considerations. Furthermore, relation (\ref{c1}) is obtained in \cite[Theorem 34]{BC} or \cite{Unt08a}. We thus focus on identity (\ref{c2}).

\smallskip

Recall that $c_2$ can be obtained as a limit of $c_2(\eta)$ when $\eta\to 0$, where $c_2(\eta)$ is given by:
\bean
c_2(\eta)&:=& 
\EX  \lc \int_{1}^{2} \lp B_s^{(1)}(\eta) - B_1^{(1)}(\eta) \rp\, d B_s^{(2)}(\eta) \int_{0}^{1} B_s^{(1)}(\eta) \, d B_s^{(2)}(\eta)\rc  \\
&=& \be\lc  \ca_{01}(\eta) \, \ca_{12}(\eta) \rc.
\eean
We can thus apply identity (\ref{eq:11}) with $N=1$, $s_1=0,t_1=1,s_2=1,t_2=2$, use expression (\ref{eq:10}) for the kernel $K$, and let $\eta\to 0$ in order to obtain:
\begin{align*}
 c_2 &=  \frac{1}{2} \gamma_H \int_{1}^{2}  \int_{0}^{1}  |s_1 - s_2|^{2H-2} \left(  s_1^{2H} - 1 -  |s_1 - s_2|^{2H} +  |s_2 - 1|^{2H}  \right)  \, d s_2 \, ds_1 \\ & := c_{2,1}  + c_{2,2}   + c_{2,3} + c_{2,4},    
 \end{align*} 
 with $ \gamma_H := H(2H-1)$. It should be noticed here that, since we are integrating on the rectangle $[0,1]\times[1,2]$, the limits as $\eta\to 0$ can be taken without much care about singularities of our kernels $[x]_{\eta}^{\beta}$ for negative $\beta$'s. Moreover, direct calculations yield
 \begin{multline*}
 c_{2,1} =  \frac{1}{2} \gamma_H \int_{1}^{2}  \int_{0}^{1}  s_1^{2H}|s_1 - s_2|^{2H-2}   \, d s_2 \, ds_1  =   \frac{H}{2}  \int_{1}^{2}    s_1^{2H}( s_1 ^{2H-1} -   |s_1 - 1|^{2H-1} ) \, d s_1\\
  =   \frac{1}{8}  (2^{4H} - 1) -  \frac{H}{2}  \int_{1}^{2}    x^{2H}  |x - 1|^{2H-1}  \, d x
 \end{multline*}
 and
 $$
 c_{2,2} =  -\frac{1}{2} \gamma_H \int_{1}^{2}  \int_{0}^{1}  |s_1 - s_2|^{2H-2}   \, d s_2 \, ds_1  =  
 -\frac{1}{2}  \EX \lc( B_2^{(1)} - B_1^{(1)}) B_1^{(1)}\rc =  -\frac{1}{4} (2^{2H} - 2).  
 $$
  Finally, we have
  \begin{multline*}
 c_{2,3} =  -\frac{1}{2} \gamma_H \int_{1}^{2}  \int_{0}^{1}  |s_1 - s_2|^{4H-2}   \, d s_2 \, ds_1 
 =   -\frac{\gamma_H}{2(4H-1)}  \int_{1}^{2}     s_1^{4H-1} -   |s_1 - 1|^{4H-1}  \, d s_1
 \\ 
 =   -\frac{2H-1}{8(4H-1)}   (2^{4H}-2) ,
 \end{multline*}
 and
 \begin{align*}
 c_{2,4} &=  \frac{1}{2} \gamma_H \int_{1}^{2}  \int_{0}^{1}  |s_2-1|^{2H}|s_1 - s_2|^{2H-2}   \, d s_2 \, ds_1 \\ & =   \frac{H}{2}  \int_{0}^{1}    |s_2 - 1|^{2H}(  |  2-s_2|^{2H-1} -   |  1-s_2|^{2H-1}) \, d s_2
 \\ & =  -\frac{1}{8} +   \frac{H}{2}  \int_{0}^{1}    |x- 1|^{2H}|  2-x|^{2H-1}  \, d x.
 \end{align*}
By the substitution $y= x-1 $ we obtain
\begin{align*}
  \int_{1}^{2}    x^{2H}|x-1|^{2H-1}  \, d x =  \int_{0}^{1}    y^{2H-1}|y+1|^{2H}  \, d y 
\end{align*}
and moreover, by setting $y= -x+1 $  we have
$$ \int_{0}^{1}    |x - 1|^{2H}|  2-x|^{2H-1}  \, d x = \int_{0}^{1}    y^{2H}|  1+y|^{2H-1}  \, d y,  $$
where these two integral expressions appear respectively in the expressions for $c_{2,1}$ and $c_{2,4}$. Hence,  putting together our elementary calculations for $c_{2,1},\ldots,c_{2,4}$, expression (\ref{c2}) follows easily.

\end{proof}

\subsection{Proof of Theorem \ref{thm:1.1}}
Recall that we have reduced our $L^2$-estimates to the evaluation of $\be[I_iI_j]$, where $I_i=\ca_{i,i+1}$. We are now ready to compute those terms, separating three different cases:

\smallskip

{\noindent {\bf (1) Diagonal terms.}}
By stationarity of the increments and thanks to Lemma \ref{lem:3.1}, we have
$$ \EX  \lc |I_i|^2\rc =  \EX \lc\left|  \ca_{01} \right|^{2}\rc =c_1.$$
So (\ref{c1}) in Lemma \ref{lem:3.1} and (\ref{eq:15}) give
\begin{align} \label{diag}
  \sum_{i=0}^{n-1} \EX \lc |J_i^n|^2 \rc  =  T^{4H} \cdot \frac{H}{2} \left( \beta(2H, 2H) + \frac{1}{4H-1} \right) \cdot n^{-4H+1}.  \end{align}

\smallskip

{\noindent \bf (2) Secondary diagonal terms.}
Using again the stationarity of the increments and Lemma \ref{lem:3.1}, we obtain
$$ 
\EX \lc I_i I_{i+1}\rc =  \EX   \lc  \ca_{01} \, \ca_{12}\rc.
$$
Hence by (\ref{c2}) in Lemma \ref{lem:3.1} it follows
\begin{align} \label{sec-diag}
 \sum_{i,j=0, |i-j|=1}^{n-1} \EX\lc J_i^n  J_{i+1}^n \rc =   2  T^{4H} \cdot c_2 \cdot  (n-1) n^{-4H}.  
 \end{align}

\bigskip

{\noindent \bf (3)  Off-diagonal terms.}  
Let us consider now the off-diagonal terms, which will induce most of the differences in the $L^2$-limit according to the value of the Hurst parameter $H$. Observe first that, as in the proof of Lemma \ref{lemma:1:Area2N}, for $|i -j|>1$ it holds:
$$
\EX \lc  \ca_{i,i+1}(\eta)\ca_{j,j+1}(\eta) \rc
= \int_{i}^{i+1} \int_{j}^{j+1}  \int_{i}^{s_1} \int_{j}^{s_2}
K'(\eta;s_1,s_2) \, K'(\eta;u_1,u_2) 
\,d u_1 \, du_2 \, ds_1 \, ds_2.
$$
Since we are now away from the diagonal, one can take safely the limit $\eta\to 0$ in the expression above, which gives:
\begin{align}\label{eq:19}
&\EX \lc  \ca_{i,i+1} \, \ca_{j,j+1} \rc \\ 
 &= H^2(2H-1)^2 \int_{i}^{i+1} \int_{j}^{j+1}  \int_{i}^{s_1} \int_{j}^{s_2} |u_1 - u_2|^{2H-2}  |s_1-s_2|^{2H-2} \,d u_1 \, du_2 \, ds_1 \, ds_2. \nonumber
\end{align}

\medskip
 
Now we have to distinguish between four cases:

\medskip

{\it \noindent (i) The case} $H <1/2$.
Equations (\ref{eq:15}) and (\ref{eq:19}) yield directly
$$ |\EX \lc J_i^n J_j^n\rc| \leq  T^{4H} \left( H(2H-1) \int_{i}^{i+1} \int_{j}^{j+1}  |s_1-s_2|^{2H-2}  \, ds_1 \, ds_2 \right)^{2} \cdot n^{-4H} . $$
Since $|i-j| >1$, the mean value theorem  gives
$$ |\EX \lc J_i^n J_j^n\rc| \leq  C n^{-4H}  |i-j-1|^{4H-4} ,$$
where $C$ is a constant depending only on $H$ and $T$.
Note that for $H<1/2$ we have
$$ \sum_{|i-j|>1} |i-j-1|^{4H-4}  < \infty $$
and so it follows 
\begin{align} \label{off_1}
   \sum_{|i-j|>1}  |\EX J_i^n J_j^n| = o(n^{-4H+1}). 
\end{align}
\medskip

{\it \noindent (ii) The case} $1/2<H<3/4$. Applying again relation (\ref{eq:15}) and the mean value theorem  to the integral on the right hand side of relation (\ref{eq:19}), we obtain
\begin{align*}
  \sum_{|i-j|>1} \EX J_i^n J_j^n &= \frac{T^{4H}}{4} H^{2}(2H-1)^{2} n^{-4H} \sum_{|i-j|>1} |i-j-\xi_{i,j}|^{2H-2} |i-j - \tilde{\xi}_{i,j}|^{2H-2}
\end{align*}
where $\xi_{i,j}, \tilde{\xi}_{i,j} \in (-1,1)$. Note now that 
$$
\sum_{|i-j|>1}  |i-j+1|^{4H-4} \leq
\sum_{|i-j|>1} |i-j-\xi_{i,j}|^{2H-2} |i-j - \tilde{\xi}_{i,j}|^{2H-2} 
\leq  \sum_{|i-j|>1}  |i-j-1|^{4H-4} .      
$$
Moreover, it is readily checked, thanks to a Taylor expansion together with the fact that
$  \sum_{|i-j|>1} |i-j-1| ^{4H-5} < \infty$, that 
$$ 
\sum_{|i-j|>1}  |i-j\pm 1|^{4H-4} =  \sum_{|i-j|>1}  |i-j|^{4H-4} + O(1).  
$$
Hence, we have
\begin{align*}
  \sum_{|i-j|>1} \EX \lc J_i^n J_j^n \rc&= \frac{T^{4H}}{4} H^{2}(2H-1)^{2} n^{-4H} \sum_{|i-j|>1} |i-j|^{4H-4} + O(n^{-4H}).
\end{align*}
Now, observe that
\begin{align*}
 \sum_{|i-j|>1} |i-j|^{4H-4} &= 2 \sum_{i=2}^{n-1} \sum_{j=0}^{i-2}  |i-j|^{4H-4} = 
 2 \sum_{i=2}^{n-1} \sum_{j=2}^{i}  j^{4H-4} 
 \\ & = 2 \sum_{j=2}^{n-1} \sum_{i=j+1}^{n-1}  j^{4H-4} = 2 \sum_{j=2}^{n-1} (n-1-j) j^{4H-4} 
 \\ & = 2n \sum_{j=2}^{n-1} j^{4H-4} -  2 \sum_{j=2}^{n-1}  j^{4H-4}    -2\sum_{j=2}^{n-1}  j^{4H-3} 
  \\ & = 2n \sum_{j=1}^{n-1} j^{4H-4} -  2 \sum_{j=1}^{n-1}  j^{4H-4}    -2\sum_{j=1}^{n}  j^{4H-3}  + O(1)
 .
\end{align*}
Let us treat those 3 terms separately: since $H>1/2$, we have $4H-3>-1$, and thus, by Riemann sums convergence, we get
$$ 
\lim_{n\to\infty}\sum_{j=1}^{n-1} \left |\frac{j}{n} \right|^{4H-3} \cdot n^{-1} = 
\int_{0}^{1} x^{4H-3} \,dx = \frac{1}{4H-2}.$$ 
It is thus easily seen that
$   n^{-4H}    \sum_{j=1}^{n-1} j^{4H-3} = O(n^{-2}). $
Moreover, since 
$$ \sum_{j=1}^{n-1} j^{4H-4} = \zeta(4-4H) +o(1),$$
where $\zeta$ stands for the usual Riemann zeta function, we have
$$ 2n^{-4H}  \sum_{j=1}^{n} nj^{4H-4}   -  2n^{-4H}  \sum_{j=1}^{n} j^{4H-4}  =  2\zeta(4H-4) \cdot n^{-4H+1} + o(n^{-4H+1}).  $$
So altogether we obtain
\begin{align} \label{off_2}
  \sum_{|i-j|>1} \EX \lc J_i^n J_j^n\rc &=   \frac{T^{4H}}{2} \zeta(4-4H) H^2(2H-1)^2\cdot n^{-4H+1}  + o(n^{-4H+1}) .
\end{align}

\medskip

{\it \noindent (iii) The case} $H=3/4$.
Proceeding as in the previous case we obtain
\begin{align*} 
 \sum_{|i-j|>1} \EX\lc J_i^n J_j^n \rc&=    \frac{T^{4H}}{4} H^{2}(2H-1)^{2} n^{-3} \sum_{|i-j|>1} |i-j- \xi_{i,j}|^{-1/2}|i-j- \tilde{\xi}_{i,j}|^{-1/2}    \\ & =  \frac{T^{4H}}{4} H^{2}(2H-1)^{2} n^{-3} \sum_{|i-j|>1} |i-j|^{-1} + O(n^{-2}).    
\end{align*}
Moreover, following again the computation of our Case (ii) above, we obtain
 \begin{align}\label{eq:23}
 \sum_{|i-j|>1} |i-j|^{-1} & = 2 \sum_{j=1}^{n-1} (n-1-j) j^{-1} + O(1) 
 = 2 \sum_{j=1}^{n-1} (n-1) j^{-1} - 2(n-1)+O(1).
\end{align}
Clearly, $2n^{-3}(n-1)=O(n^{-2})$. Moreover, since $ \sum_{j=1}^{n-1} j^{-1} = c+ \log(n)+o(1)$, where $c$ stands for the Euler-Mascheroni constant, we get
$$ 2n^{-3} \sum_{j=1}^{n-1}( n -1) j^{-1}  = 2n^{-2} \log(n) + O(n^{-2}). $$
Hence, plugging these two relations into equation (\ref{eq:23}), it follows
\begin{align} \label{off_3}
  \sum_{|i-j|>1} |\EX J_i^n J_j^n| &=  \frac{T^{4H}}{2} H^{2}(2H-1)^{2}  \log(n) n^{-2}  + O(n^{-2}) .
\end{align}

\medskip

{\it \noindent (iv) The case} $H>3/4$.
Along the same lines as in the previous cases, we end up with:
\begin{align*} 
 \sum_{|i-j|>1} \EX J_i^n J_j^n &= \frac{T^{4H}}{4} H^{2}(2H-1)^{2} n^{-4H} \sum_{|i-j|>1} |i-j- \xi_{i,j}|^{2H-2}|i-j- \tilde{\xi}_{i,j}|^{2H-2}    \\ & =  \frac{T^{4H}}{4} H^{2}(2H-1)^{2} n^{-4H} \sum_{|i-j|>1} |i-j|^{4H-4} + o(n^{-2}).    
\end{align*}
Since $4H-4>-1$, we now obtain:
$$ 
\sum_{|i-j|>1} \left| \frac{i-j}{n} \right|^{4H-4} \cdot n^{-2} \longrightarrow    \int_0^1  \int_0^1  |x-y|^{4H-4} \, dx \, d y= \frac{2}{(4H-3)(4H-2)},
$$
which yields, for any $H>3/4$:  
\begin{align} \label{off_4} 
 \sum_{|i-j|>1} \EX \lc J_i^n J_j^n\rc  
 = \frac{T^{4H}}{4} \frac{H^{2}(2H-1)}{4H-3} \cdot n^{-2} + o(n^{-2}) 
\end{align}

\medskip 

Theorem \ref{thm:1.1} now follows easily from combining (\ref{diag}), (\ref{sec-diag}), (\ref{off_1}), (\ref{off_2}), (\ref{off_3}) and (\ref{off_4}).

\subsection{Proof of Theorem \ref{milstein}}

Recall that the Milstein-type scheme is given by
$$ \widehat{X}_T^n= \frac{1}{2}  \sum_{i=0}^{n-1} \big( B_{iT/n}^{(1)} + B_{(i+1)T/n}^{(1)} \big)\big( B_{(i+1)T/n}^{(2)}  -   B_{iT/n}^{(2)} \big). $$
As in the proof of Lemma \ref{lem:3.1}, using  the scaling property,  the self-similarity of fBm, Lemma 2.6 and moreover letting $\eta \rightarrow 0$ and applying dominated convergence (note that we assume here $H>1/2$), the mean square error of the Milstein-type scheme 
is given by 
$$      n^{-4H}  T^{4H} \gamma_H   \sum_{i=0}^{n-1} \sum_{j=0}^{n-1} \int_{i}^{i+1} \int_{j}^{i+1}  \theta_{i,j}(s_{1},s_{2})   |s_1-s_2|^{2H-2}\, ds_{2} \, ds_{1}
 $$
with $\gamma_H=H(2H-1)$ and
\begin{align*}
 & \theta_{i,j}(s_{1},s_{2})   = \frac{1}{4} \EX ( 2B_{s_1}^{(1)} -B_{i}^{(1)} - B_{i+1}^{(1)})( 2B_{s_2}^{(1)} -B_{j}^{(1)}-B_{j+1}^{(1)} )
\end{align*}
for $s_{1} \in [i,i+1]$, $s_{2} \in [j,j+1]$, $i,j = 0, \ldots, n-1.$

\smallskip

\noindent
{\it (i)} We first show that the contribution of the off-diagonal terms  to the error is 
 asymptotically negligible, i.e., 
\begin{align}\label{off-diag-m}
& n^{-4H} \left| \sum_{|i-j| > \log(n)}\int_{i}^{i+1} \int_{j}^{j+1} \theta_{i,j}(s_{1},s_{2})  |s_1-s_2|^{2H-2} \, ds_{2} \, ds_{1}  \right| \\ &= 2n^{-4H}  \left| \sum_{i-j > \log(n)}\int_{i}^{i+1} \int_{j}^{j+1} \theta_{i,j}(s_{1},s_{2})  |s_1-s_2|^{2H-2} \, ds_{2} \, ds_{1} \right| \nonumber
= o(n^{-4H+1}).\end{align}
In \cite{AN} (see Appendix A) it is shown that
\begin{align} \label{off-diag-stein}
 \left|  \sum_{i-j > \log(n)}\int_{i}^{i+1} \int_{j}^{j+1} \theta_{i,j}(s_{1},s_{2})   \, ds_{2} \, ds_{1}  \right| \leq C \cdot (\log(n))^{4H-2} \cdot n^{2H-2}.\end{align}

 To use this estimate,  define now
\begin{align*}
R_{i,j}^1= \{ s_1,s_2 \in [i,i+1] \times [j,j+1]: \, \theta_{i,j}(s_1,s_2) \geq 0 \} ,\\
R_{i,j}^2=  \{ s_1,s_2 \in [i,i+1] \times [j,j+1]: \, \theta_{i,j}(s_1,s_2) < 0 \}.
\end{align*} 
An application of the mean value theorem gives
\begin{align*}
&    |i-j+1|^{2H-2} \int \int_{R_{i,j}^1} \theta_{i,j}(s_{1},s_{2})   \, ds_{2} \, ds_{1} +   |i-j-1|^{2H-2} \int \int_{R_{i,j}^2} \theta_{i,j}(s_{1},s_{2})   \, ds_{2} \, ds_{1} \\
 &   \, \, \leq \int_{i}^{i+1} \int_{j}^{i+1} \theta_{i,j}(s_{1},s_{2})  |s_1-s_2|^{2H-2} \, ds_{2} \, ds_{1}\\
 & \quad  \leq   |i-j-1|^{2H-2} \int \int_{R_{i,j}^1} \theta_{i,j}(s_{1},s_{2})   \, ds_{2} \, ds_{1}  +   |i-j+1|^{2H-2} \int \int_{R_{i,j}^2} \theta_{i,j}(s_{1},s_{2})   \, ds_{2} \, ds_{1}.
\end{align*} 
Note that
$$  \left| \int \int_{R_{i,j}^1} \theta_{i,j}(s_{1},s_{2})   \, ds_{2} \, ds_{1} \right| + \left| \int \int_{R_{i,j}^2} \theta_{i,j}(s_{1},s_{2})   \, ds_{2} \, ds_{1} \right|  \leq 2 , $$
so it follows
\begin{align*}
 & n^{-4H} \int_{i}^{i+1} \int_{j}^{i+1} \theta_{i,j}(s_{1},s_{2})  |s_1-s_2|^{2H-2}\, ds_{2} \, ds_{1}\\
 &  \qquad = n^{-4H} |i-j|^{2H-2} \int_{i}^{i+1} \int_{j}^{i+1} \theta_{i,j}(s_{1},s_{2})   \, ds_{2} \, ds_{1} + n^{-4H} \rho_{i,j}
\end{align*} 
with
$$  |\rho_{i,j} | \leq C \cdot |i-j-1|^{2H-3}.$$
Using (\ref{off-diag-stein}) we thus have
\begin{align*} 
&  n^{-4H}  \left|\sum_{|i-j| > \log(n)}\int_{i}^{i+1} \int_{j}^{j+1} \theta_{i,j}(s_{1},s_{2})  |s_1-s_2|^{2H-2}    \, ds_{2} \, ds_{1} \right| \\ & \qquad \qquad \quad  \quad  \leq C  \cdot (\log(n)^{6H-4}n^{-2H-2}) + \sum_{|i-j| > \log(n)} n^{-4H} |\rho_{i,j}|.
 \end{align*}
Since
\begin{align*} 
  \sum_{|i-j| > \log(n)}    |i-j|^{2H-3} \leq  n \sum_{i > \log(n)}^{\infty}   i^{2H-3} = O( n \log(n)^{2H-2}),
  \end{align*}
  we have obtained
 \begin{align*}
n^{-4H}  \left| \sum_{|i-j| > \log(n)}\int_{i}^{i+1} \int_{j}^{j+1} \theta_{i,j}(s_{1},s_{2})  |s_1-s_2|^{2H-2} \, ds_{2} \, ds_{1} \right|= o(n^{-4H+1}),\end{align*} that is (\ref{off-diag-m}).

\smallskip

\noindent
{\it (ii)} Now consider the "close to  diagonal" terms. Since 
\begin{align*}\theta_{i,j}(s_1,s_2)&= \frac{1}{4} \gamma_H \int_i^{i+1} \int_j^{j+1} \big (1_{[i,s_1]}(u_1)-1_{[s_1,i+1]}(u_1) \big) \\ & \qquad \qquad \qquad \qquad \qquad \times \big(1_{[i,s_2]}(u_2)-1_{[s_2,i+1]}(u_2) \big) |u_1-u_s|^{2H-2} \, du_2 du_1, \end{align*}
we have  for $|i-j|>1$ that
 $$ |\theta_{i,j}(s_1,s_2)| \leq C \cdot |i-j-1|^{2H-2}.$$
 Thus it follows
 \begin{align*} \left|  \sum_{1<|i-j| < \log(n)}\int_{i}^{i+1} \int_{j}^{j+1} \theta_{i,j}(s_{1},s_{2})  |s_1-s_2|^{2H-2} \, ds_{2} \, ds_{1} \right| &\leq C \sum_{1<|i-j| < \log(n)} |i-j-1|^{4H-4} \\ & \leq C  \log(n) \sum_{i=1}^{n}  i^{4H-4}.
 \end{align*}
 If $H<3/4$ then
 $$  \sum_{i=1}^{n}  i^{4H-4} < \infty. $$
 Moreover, if $H=3/4$ then
 $$  \sum_{i=1}^{n}  i^{4H-4} = c+\log(n) + o(1),$$
 where $c$ is again the Euler-Mascheroni constant.
 Finally, if $H>3/4$ we have
$$ \sum_{i=1}^{n}  (i/n)^{4H-4} \cdot n^{-1} \longrightarrow \int_0^1 x^{4H-4} dx = \frac{1}{4H-3} $$
and so
$$ \sum_{i=1}^{n}  i^{4H-4} = O(n^{4H-3}). $$
Hence we obtain
 \begin{align*} n^{-4H}  \left| \sum_{1<|i-j| < \log(n)}\int_{i}^{i+1} \int_{j}^{j+1} \theta_{i,j}(s_{1},s_{2})  |s_1-s_2|^{2H-2} \, ds_{2} \, ds_{1} \right| & =o(n^{-4H+1}).
 \end{align*}

\smallskip

\noindent
{\it (iii)} Combining step (i) and (ii) we have
$$      n^{-4H} \left| \sum_{1<|i-j| < n}\int_{i}^{i+1} \int_{j}^{j+1} \theta_{i,j}(s_{1},s_{2})  |s_1-s_2|^{2H-2} \, ds_{2} \, ds_{1} \right|=o(n^{-4H+1}).   $$
Therefore, it follows that
the leading error term of the Milstein-type scheme is given by
\begin{align*}
    & \gamma_H T^{4H} \sum_{0 \leq |i-j| \leq 1 } \int_{i}^{i+1} \int_{j}^{j+1} \theta_{i,j}(s_{1},s_{2})  |s_1-s_2|^{2H-2} \, ds_{2} \, ds_{1} \\ & \quad
= T^{4H}\sum_{i=0}^{n-1}  \EX \left| \int_{i}^{i+1} \Big( B_s^{(1)} - \frac{1}{2}(B_{i}^{(1)} + B_{i+1}^{(1)}) \Big) d B_s^{(2)}\right|^2  \\ & \qquad + 
T^{4H} \sum_{|i-j|=1}  \EX  \int_{i}^{i+1} \Big( B_{s_1}^{(1)} - \frac{1}{2}(B_{i}^{(1)} + B_{i+1}^{(1)}) \Big) d B_{s_1}^{(2)}  \int_{j}^{j+1} \Big( B_{s_2}^{(1)} - \frac{1}{2}(B_{j}^{(1)} + B_{j+1}^{(1)}) \Big) d B_{s_2}^{(2)}.
\end{align*}
Using again the scaling and self-similarity property of fBm we obtain
$$   \lim_{n \rightarrow \infty} n^{4H-1}  \sum_{0 \leq |i-j| \leq 1 } \int_{i}^{i+1} \int_{j}^{j+1} \theta_{i,j}(s_{1},s_{2})  |s_1-s_2|^{2H-2} \, ds_{2} \, ds_{1} = \alpha_4(H)$$
with 
$$
\alpha_4(H)= \frac{1}{4}  \EX  \int_1^2  (2B_{s_1} ^{(1)} -B_{1}^{(1)}-B_{2}^{(1)}) \, d B_{s_1} ^{(2)}  \int_0^3 (B_{s_2} ^{(1)}-B_{1}^{(1)}-B_{2}^{(1)})  \, d B_{s_2} ^{(2)},$$ 
which is the assertion of Theorem {\ref{milstein}}.

\bigskip
\bigskip


\section{Asymptotic error distribution of the Euler scheme: $H<3/4$}\label{sec:fourth-moments}

 Let us first explain the strategy we have adopted in order to obtain our central limit theorem for the difference $X_T-X_T^n$ of the Euler scheme and its approximation  in the case $H<3/4$. First,  recall that the random variable $X_T-X_T^n$ can be expressed as 
$$
X_T - X_T^n=\sum_{i=1}^n  J_{i}^{n},
\quad\mbox{with}\quad
J_{i}^{n}\triangleq \int_{iT/n}^{(i+1)T/n} (B_s^{(1)} - B_{i/n}^{(1)}) \, d B_s^{(2)}.
$$
With this expression in hand, it can be seen in particular that $X_T - X_T^n$ is still an element of the second chaos of our underlying fBm $B$.

\smallskip

Let us then recall the following  limit theorem for random variables in a fixed finite Gaussian  chaos,  which can be found in \cite[Theorem 1]{NP}:
\begin{prop}
Fix $p\ge 1$.
Let $\{Z_n;\, n\ge 1\}$ be a sequence of centered random variables belonging the $p\textsuperscript{th}$  chaos of a Gaussian process, and assume that 
\begin{equation}\label{eq:2:cond1}
\lim_{n\to\infty}\be[Z_n^2]=1.
\end{equation}
Then $Z_n$ converges in distribution to a centered Gaussian random variable if and only if  the following condition is met:
\begin{equation} \label{eq:2:cond2}
\lim_{n\to\infty}\be[Z_n^4]=3.
\end{equation}
\end{prop}
This is the criterion we shall adopt in order to get our central limit theorem.  The second order  condition (\ref{eq:2:cond1}) is simply a normalization step,  so that the essential point is to analyze the fourth order moments of $X_T - X_T^n$ in order to prove condition (\ref{eq:2:cond2}). It should be stressed at this point that \cite[Theorem 1]{NP} contains in fact a series of equivalent statements for condition (\ref{eq:2:cond2}), based either on assumptions on the Malliavin derivatives of the random variables $Z_n$, or on purely deterministic criterions concerning the kernels defining the multiple integrals under consideration. These alternative criterions yield arguably some shorter computations, but we preferred to stick to the fourth order moment for two main reasons: \textit{(i)} The computations we perform in this context are more intuitive, and in a sense, easier to follow. \textit{(ii)} As we shall explain below, the fourth order computations lead to some visual representations in terms of graphs, and we will able to show easily that the CLT is equivalent to have \emph{the sum of the connected diagrams tending to 0}. As we shall see, this latter criterion is really analogous to \cite[Theorem 1, Condition (ii)]{NP}.

\smallskip

 In the remainder of this section, we check condition (\ref{eq:2:cond2}) for $X_T-X_T^n$, rescaled according to Theorem \ref{thm:1.1}, in order to get a central limit theorem  for our approximation. We shall first explain the basics of our diagrammatical method of computation and show how to reduce our problem to the analysis of connected diagrams. Then we split our study into regular and singular terms.

\subsection{Reduction of the problem}
Owing to Theorem \ref{thm:1.1}, it is enough for our purposes to show that $\lim_{n\to\infty}\be[Z_n^4]=3$, where 
\begin{equation}\label{eq:def-Zn}
Z_n=n^{2H-1/2} T^{-2H}  \lc \al_{\ell}(H)\rc^{-1/2}\sum_{i=1}^n J_i^n,
\end{equation} 
and where the index $\ell$ varies in $\{1,2\}$ according to the value of $H$. Furthermore, the self-similarity of fBm implies that 
$$
\be [Z_n^4]=  (\al_\ell(H) n)^{-2} \,  \be \left[ \left( \sum_{i=1}^n I_i \right)^4 \right],
$$ 
where $I_i={\mathcal A}_{i,i+1}$ is the L\'evy area between $i$ and $i+1$. Now, the most naive idea one can have in mind is to write $Z_n$ as $\lim_{n\to\infty}Z_n(\eta)$, where $Z_n$ is obtained by considering regularized areas based on $B(\eta)$, and then expand $\be [Z_n^4(\eta)]$ as
\begin{align}\label{eq:exp-Zn4}
&\be [Z_n^4(\eta)]=(\al_ \ell(H) n)^{-2} \, \sum_{i_1,\ldots,i_4=1}^n  
\EX  \lc \prod_{j=1}^{4} I_{i_j}(\eta) \rc 
\\
&=(\al_\ell(H) n)^{-2} \, \prod_{j=1}^{4}\lp \int_{i_j}^{i_j+1}dx_j\int_{i_j}^{x_j} dy_j \rp
\be\lc  \prod_{j=1}^{4} B_{x_j}^{'(1)}(\eta)\rc \, \be\lc  \prod_{j=1}^{4} B_{y_j}^{'(2)}(\eta)\rc, \notag
\end{align}
where we have used formula (\ref{eq:13}) with $N=2$ in order to get the last equality.

\smallskip

We apply now Wick's formula (\ref{eq:11}) in order to get an expression for the expected values above, and this is where our diagrammatical representation can be useful. Indeed, $\be[  \prod_{j=1}^{4} B_{x_j}^{'(1)}(\eta)] \, \be[  \prod_{j=1}^{4} B_{y_j}^{'(2)}(\eta)]$  is the sum of 9 different terms, connecting the $x_i$'s two by two according to formula (\ref{eq:11}), and also the $y_i$'s two by two.
 Each term may be represented by a four-point diagram in the following way. Draw a simple line, resp. a dashed line between $i$ and $j$ if $x_i$ and $x_j$, resp. $y_i$ and $y_j$  are connected. This procedure yields 9 different graphs, whose typical examples are given at Figure \ref{Fig2}.
 \begin{figure}[h]
   \centering
   \includegraphics[scale=0.6]{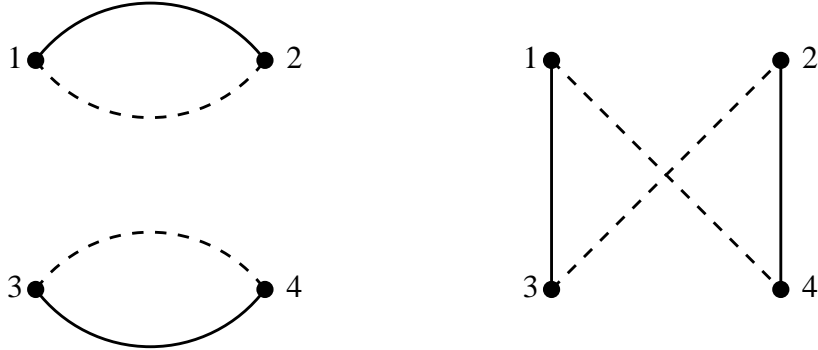}
    \caption{Two examples of diagrams.}
   \label{Fig2}
 \end{figure}
Moreover, the reader can then check easily that diagrams fall into two types: connected ones (6) and disconnected ones (3). Furthermore, up
to permutations of the indices, there is only one disconnected diagram, namely the first diagram of Figure \ref{Fig2}.  One checks immediately that the corresponding
integral  is $\be [I_{i_1}(\eta) I_{i_2}(\eta)] \be [I_{i_3}(\eta) I_{i_4}(\eta)]$ . Write also the
total contribution of the 6 {\it connected} diagrams as $\be [I_{i_1}(\eta) I_{i_2}(\eta)I_{i_3}(\eta) I_{i_4}(\eta)]_{(c)}$. Thanks to our graphical representation, it is then straightforward to prove  the following: for arbitrary 
constants $c_i$, $i=1,\ldots,n$, we have
\begin{equation} \label{eq:33}\be \lc\left( \sum_{i=1}^n c_i I_i(\eta)\right)^4\rc - 3\,   \EX^2 \lc \left( \sum_{i=1}^n c_i I_i(\eta)
\right)^2\rc=  \be \left[ \left( \sum_{i=1}^n c_i I_i(\eta)\right)^4 \right]_{(c)}. \end{equation}
Hence our condition (\ref{eq:2:cond2}) is satisfied for $Z_n$ defined by (\ref{eq:def-Zn}) if and only if the right-hand side of equation (\ref{eq:33}) goes to zero for $c_i=n^{-1/2}$ ($c_i$ is in fact independent of $i$). It should be stressed at that point that the latter condition (which is what we call \textit{connected diagrams go to 0}) is an analog of   criterion (ii) in \cite[Theorem 1]{NP}, but is obtained here without Malliavin calculus tools. This terminology
 is  inspired by the Feynman diagram analysis in the context of quantum field theory, see e.g. \cite{LeB}.

\smallskip

Let us set now $\tilde Z_n(\eta)=\sum_{i=1}^n I_i(\eta)$. With the above considerations in mind, we are reduced to show that 
\beq\label{eq:34}
\lim_{n\to\infty} \lim_{\eta\to 0}\frac{1}{n^2} 
\be \left[ \tilde Z_n^4(\eta)  \right]_{(c)} = 0.
\eeq
This relation will be first proved  for $H\in(1/2,3/4)$. In that case one may consider directly the situation where $\eta=0$, that is the infinitesimal covariance kernel $(x,y)\mapsto H(2H-1)
|x-y|^{2H-2}$, since
it is locally integrable. The proof requires only a few lines of computations. Each diagram in $\be [\tilde Z_n^4(\eta) ]_{(c)}$ splits into
{\em regular terms} -- which are also well-defined for $H<\half$ -- and {\em singular terms} -- which diverge when $H<\half$. As we 
shall see, the bounds given for the non-singular terms also hold true for $H<\half$. Then we shall see how to bound the singular terms
for arbitrary $H$ by replacing the ill-defined kernel $H(2H-1) |x-y|^{2H-2}$ with its regularization $K'(\eta;x,y)$. This step is of course only needed
in the case $H<\half$, but computations are equally valid in the whole range $H\in(1/4,3/4)$. In other words, the barrier $H=\half$ is
largely artificial (the proofs of the two cases are actually mixed, and one could also have written a general proof, at the price
of some more  technical calculations).

\smallskip

Before we enter into these computational details, let us reduce our problem a little bit more: recall again that we wish to prove relation (\ref{eq:34}) for $\tilde Z_n(\eta)=\sum_{i=1}^n I_i(\eta)$. As explained above, we evaluate $\be[ \tilde Z_n^4(\eta) ]_{(c)}$ with 6 different connected diagrams. Let us focus on the term, which will be called $\ct$, corresponding to the diagram given at Figure~\ref{Fig3} (the other ones can be treated in a similar manner).
 \begin{figure}[h]
   \centering
   \includegraphics[scale=0.6]{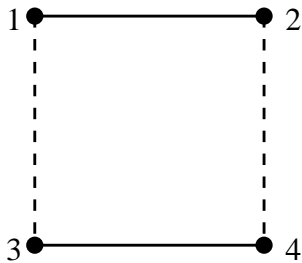}
    \caption{Typical connected diagram.}
   \label{Fig3}
 \end{figure}
Now, starting from expression (\ref{eq:exp-Zn4}), taking into account the fact that we are considering the particular diagram given at Figure \ref{Fig3} and integrating over the internal variables $y$, we end up with
$\ct= n^{-2} \,  \sum_{i_1,\ldots,i_4=1}^n  I_{(i_1,\ldots,i_4)} $,
where (recalling that the kernel $K$ is defined by equation (\ref{eq:10a}))
\begin{multline}\label{eq:36}
 I_{(i_1,\ldots,i_4)} := \int_{i_1}^{i_1+1} dx_1 \cdots  \int_{i_4}^{i_4+1} dx_4 \\
 K'(\eta;x_1,x_2) K'(\eta;x_3,x_4) K_{i_1,i_3}(\eta;x_1,x_3) K_{i_2,i_4}(\eta;x_2,x_4).
\end{multline}

\smallskip

The latter expression yields naturally a  notion of {\it regular terms} and {\it singular terms}:
split the set of indices $\{1,\ldots,n\}^4$ into $A_1\cup A_2$, where
\beq A_2= \{(i_1,\ldots,i_4)\ |\ 1\le i_1,\ldots,i_4\le n,\   |i_1-i_3|, |i_2-i_4|\le 1\},\quad A_1=\{1,\ldots,n\}^4 \setminus A_2. 
\eeq
Regular terms, resp. singular terms are those for which $|i_1-i_2|, |i_3-i_4|\ge 2$, resp. $|i_1-i_2|\le 1$ or $|i_3-i_4|\le 1$.
Split accordingly the sets of indices $A_j$, $j=1,2$ into $A_{j,{\rm reg}}\cup A_{j,{\rm sing}}$, and denote
\begin{equation}\label{eq:def-treg-tsing}
\ct_{j,{\rm reg}}  = \sum_{(i_1,\ldots,i_4)\in A_j^{{\rm reg}}} I_{(i_1,\ldots,i_4)} 
\quad\mbox{and}\quad
\ct_{j,{\rm  sing}} =  \sum_{(i_1,\ldots,i_4)\in A_j^{{\rm sing}}} I_{(i_1,\ldots,i_4)}.
\end{equation}
It remains to prove that $\ct_{j,{\rm reg}} =o(n^2)$ and $ \ct_{j,{\rm sing}}=o(n^2)$, for $j=1,2$. These two steps will be performed respectively at Section \ref{sec:reg-terms} and \ref{sec:sing-terms}.

\subsection{Regular terms and case $\mathbf{H>1/2}$}\label{sec:reg-terms}
This section is devoted to the study of $\ct_{j,{\rm reg}}$, and also of $\ct_{j,{\rm  sing}}$ for $H>1/2$. In both cases, one is allowed to take limits as $\eta\to 0$ without much care, by a standard application of the dominated convergence theorem. We skip this elementary step, and consider directly our expressions for $\eta=0$. 

\smallskip

Let us start by $\ct_{1,{\rm reg}}$, which is given by
\begin{multline}\label{eq:38}
\ct_{1,{\rm reg}} = 
\sum_{|i_1-i_3|,|i_1-i_2|,|i_3-i_4|\ge 2} \int_{i_1}^{i_1+1} dx_1 \cdots  \int_{i_4}^{i_4+1} dx_4 \,
K'(x_1,x_2) K'(x_3,x_4) \\ 
K_{i_1,i_3}(x_1,x_3) K_{i_2,i_4}(x_2,x_4).
\end{multline}
We shall bound this integral by different methods in the cases $H\in(1/2,3/4)$ and $H<\half$:

\smallskip

\noindent
\textit{(i)}
Assume first $H\in(1/2,3/4)$.  Whenever $|s-i|,|t-j|\le 1$,  recall from Lemma \ref{lem:K} that $K_{i,j}(s,t)\lesssim |t-s|^{2H-2}$ if $|i-j|\ge 2$, and $s\in[i,i+1], t\in[j,j+1]$. In particular, the quantity $|K_{i_1,i_3}(x_1,x_3)|$ in equation (\ref{eq:38}) is bounded by $|x_1-x_3|^{2H-2}$. We also obviously have $|K'(x_1,x_2)|\lesssim |x_2-x_1|^{2H-2}$ and $|K'(x_3,x_4)|\lesssim |x_4-x_3|^{2H-2}$. As a consequence,
\begin{multline*}
 |\ct_{1,{\rm reg}}| \leq 2 C \sum_{|i_1-i_3|,|i_1-i_2|,|i_3-i_4|\ge 2} \int_{i_1}^{i_1+1} dx_1 \cdots  \int_{i_4}^{i_4+1} dx_4 \,
|x_2-x_1|^{2H-2} |x_4-x_3|^{2H-2} \\  
\times |x_1-x_3|^{2H-2} |K_{i_2,i_4}(x_2,x_4)|. 
\end{multline*}
Let us undo now the initial scaling by setting $t_j=x_j/n$. One gets
\begin{multline} 
|\ct_{1,{\rm reg}}|  \lesssim n^{4+3(2H-2)} \int_0^1 dt_1\cdots\int_0^1 dt_4 \, |t_2-t_1|^{2H-2} |t_4-t_3|^{2H-2} \\
|t_3-t_1|^{2H-2} K_{\lfloor
nt_2\rfloor, \lfloor nt_4\rfloor}(nt_2,nt_4).  \label{eq:3:power-integral} 
\end{multline}
Applying Lemma \ref{lem:2:alphabetaab} to the above expression  (\ref{eq:3:power-integral}) and integrating successively with respect to $t_1$ and $t_3$
yields
 \beq |\ct_{1,{\rm reg}}|\lesssim n^{4+3(2H-2)} \int_0^1 dt_2 \int_0^1 dt_4 (1+|t_2-t_4|^{6H-4})  K_{\lfloor
nt_2\rfloor, \lfloor nt_4\rfloor}(nt_2,nt_4) . \eeq
Recall now that $|K_{\lfloor nt_2\rfloor, \lfloor nt_4\rfloor}(nt_2,nt_4)|\lesssim \min( 1,(n|t_2-t_4|)^{2H-2})$. Hence, one can bound this kernel by 1 on $[0,1/n]$ and by $(nt)^{2H-2}$ on $[1/n,1]$, yielding
\beq \int_0^1 dt_4 K_{\lfloor nt_2\rfloor,\lfloor nt_4\rfloor}(nt_2,nt_4)\lesssim \int_0^{1/n}\ dt+n^{2H-2} \int_{1/n}^1 t^{2H-2}\ dt
\lesssim n^{-1}+n^{2H-2}, \eeq
and also
\bean    
\int_0^1 dt_4 |t_2-t_4|^{6H-4}  K_{\lfloor
nt_2\rfloor, \lfloor nt_4\rfloor}(nt_2,nt_4) &\lesssim&  \int_0^{1/n} t^{6H-4}dt+n^{2H-2}\int_{1/n}^1 t^{8H-6} dt  \\
&\lesssim& n^{3-6H}+n^{2H-2}. 
\eean
Hence one has found:  $ |\ct_{1,{\rm reg}}|\lesssim n+n^{8H-4}+n^{6H-3}\lesssim n+n^{8H-4}.$ In particular, if
 $H<3/4$, then $|\ct_{1,{\rm reg}}|=o(n^2).$

\smallskip

\noindent
\textit{(ii)}
Assume now $H<\half$. In this case, the integrals we have been manipulating above are divergent, so that we will use series arguments instead. Let us observe then that, under the same conditions as in the case $H\in(1/2,3/4)$, the bound $|K_{i_1,i_3}(x_1,x_3)|\lesssim |i_1-i_3|^{2H-2}$ holds true. We also bound the factor $ |K_{i_2,i_4}(x_2,x_4)|$ by a constant in order to get
\begin{align*} 
&  |\ct_{1,{\rm reg}}|\\
&\lesssim\sum_{i_1,i_3:|i_1-i_3|\ge 2} |i_1-i_3|^{2H-2} 
\left( \sum_{i_2: |i_2-i_1|\ge 2} |i_2-i_1|^{2H-2} \right) \left( \sum_{i_4: |i_4-i_3|\ge 2} |i_4-i_3|^{2H-2} \right) \\
&\lesssim\sum_{i_1,i_3:|i_1-i_3|\ge 2} |i_1-i_3|^{2H-2} =O(n). 
\end{align*}

\smallskip

We now leave to the reader the task of checking, with the same kind of computations, that $|\ct_{1,{\rm sing}}|=O(n)$ ({\it provided $H>\half$}).
\bigskip

\smallskip

Turn now to the complementary set of indices, $A_2$: by simply bounding the kernels $K_{i,j}(x,y)$ by constants in  (\ref{eq:38}), one gets
\begin{eqnarray}   |\ct_{2,{\rm reg}}|&\lesssim& \sum_{|i_1-i_3|,|i_2-i_4|\le 1; |i_1-i_2|,|i_3-i_4|\ge 2}   \int_{i_1}^{i_1+1} dx_1 \cdots  \int_{i_4}^{i_4+1} dx_4 
|K'(x_1,x_2)| |K'(x_3,x_4)| \nonumber\\
& \lesssim&\sum_{i_1,i_2:|i_1-i_2|\ge 2} |i_1-i_2|^{2(2H-2)}. \label{eq:3:T2reg} \end{eqnarray}
Hence $|\ct_{2,{\rm reg}}|=O(n^{4H-2})=o(n^2)$ when $H<3/4$, which is enough for our purposes.

\smallskip

Finally, {\it provided} $H>\half$, some similar elementary considerations prove  that 
\beq |\ct_{2,{\rm sing}}|\lesssim n\left( \int_0^1 dx_1\int_0^1 dx_2 |K'(x_1,x_2)| \right)^2=O(n),
\eeq
where we have used the fact that $|i_j-i_k|=O(1)$ for $j,k=1,\ldots,4$ if
$(i_1,\ldots,i_4)\in A_{2,{\rm sing}}$.

\subsection{Singular terms in the case $\mathbf{H<\half}$}\label{sec:sing-terms}

Let us reconsider the terms $\ct_{1,{\rm sing}}$ and $\ct_{2,{\rm sing}}$  in (\ref{eq:def-treg-tsing}), taking now into account the fact that we deal  with the regularized kernels $K'(\eta;x_1,x_2)$, $K'(\eta;x_3,x_4)$ instead of $K'(x_1,x_2)$, $K'(x_3,x_4)$. 

\smallskip

In order to treat all the terms appearing in our sums in a systematic way, let us introduce a little of vocabulary: consider any multi-index $(i_1,\ldots,i_p)$, $p\ge 2$ (in our case $p=4$). We shall  say that $\{i_{j_1},\ldots,i_{j_k}\}$, $j_1\not=\ldots\not=j_k$
is a {\it maximal contiguity subset} of $(i_1,\ldots,i_p)$ if (up to a reordering) $i_{j_2}-i_{j_1}=\ldots=i_{j_k}-i_{j_{k-1}}=1$
and $i_l\ge i_{j_k}+2$ or $\le i_{j_1}-2$ if $l\not=j_1,\ldots,j_k$. Maximal contiguity subsets define a partition of the set $\{i_1,\ldots,i_p\}$.
Then we shall write $(i_1,\ldots,i_p)\in J_{m_1,\ldots,m_q}$ if the lengths of the maximal contiguity subsets of $(i_1,\ldots,i_p)$ 
are $m_1\ge\ldots\ge m_q$.

\smallskip

This terminology will help us classify the terms in $\ct_{1,{\rm sing}}\cup \ct_{2,{\rm sing}}$. Forgetting about the $O(n)$ multi-indices
$(i_1,\ldots,i_4)$ in $J_4$ appearing in $\ct_{2,{\rm sing}}$  (according to the fact that ${\rm Var}(\ca_{st}(\eta))$ is uniformly bounded on $[0,T]$, proved in \cite{Unt08a}, this term contributes only $O(n)$ to the sum), the other singular terms are all in
$\ct_{1,{\rm sing}}$ and may be:

-- either of type $J_{2,1,1}$, with maximal contiguity subsets $\{ \{i_1,i_2\}, \{i_3\}, \{i_4\}\}$ or equivalently
 $\{ \{i_3,i_4\}, \{i_1\}, \{i_2\}\}$;

-- or of type $J_{2,2}$,  with maximal contiguity subsets $\{ \{i_1,i_2\}, \{i_3,i_4\}\}$;

-- or of type $J_{3,1}$, with maximal contiguity subsets $\{ \{i_1,i_2,i_3\}, \{i_4\}\}$ or equivalent possibilities.

\smallskip

Let us observe that, in our iterated multiple integrals,  the most serious problems of singularity appear when the external variables $x$ (represented by solid lines in our graphs) are contiguous. Indeed, the internal variables $y$ are integrated, smoothing the kernels $K'$ into $K_{a,b}$. However, one still has to cope with the highly singular kernel $K'$ for the external variables. For instance, for the graph given at Figure \ref{Fig3} (which is the one we are analyzing), this kind of problem appear for the terms of type $J_{2,1,1}$ (when the maximal contiguity subset is $\{ \{i_1,i_2\}, \{i_3\}, \{i_4\}\}$) or $J_{2,2}$.
But a simple Fubini type argument allows us to get rid of these singularities. Indeed, when $\eta>0$, the integral
$$ 
\prod_{j=1}^{4} \int_{i_j}^{i_{j}+1} \!\!\! dx_j  \ K'(\eta;x_1,x_2) K'(\eta;x_3,x_4) \ .\
\prod_{j=1}^{4} \int_{i_j}^{x_j} \!\!\!dy_j \ K'(\eta;y_1,y_3) K'(\eta;y_2,y_4), 
$$
corresponding to the diagram of Figure \ref{Fig3}, is also equal to 
$$
\prod_{j=1}^{4} \int_{i_j}^{i_{j}+1} \!\!\! dy_j \ K'(\eta;y_1,y_3) K'(\eta;y_2,y_4) \ .\
\prod_{j=1}^{4} \int_{y_j}^{i_{j}+1} \!\!\! dx_j  \ K'(\eta;x_1,x_2) K'(\eta;x_3,x_4),
$$
corresponding (up to time-reversal) to the reversed diagram obtained by exchanging full lines with dashed lines. The important point is that this \textit{full-line dashed-line symmetry} maps the above singular
diagrams of type $J_{2,1,1}$ or $J_{2,2}$ into regular diagrams, for which the external variables are separated. This situation can thus be handled along the same lines as in Section \ref{sec:reg-terms}, and there only remains to estimate singular diagrams of type 
$J_{3,1}$. 

\smallskip

For this latter class of diagram, assume for instance (without loss of generality) that $\{i_1,i_2,i_3\}$ is a maximal contiguity subset of $(i_1,\ldots,i_4)$. Then, owing to relation (\ref{eq:10}), the corresponding integral writes $E=E(i_1,\ldots,i_4)$, with
\begin{multline}\label{eq:44}
E= c_H \int_{i_3}^{i_3+1} dx_3 \int_{i_1}^{i_1+1} dx_1 \int_{i_2}^{i_2+1} dx_2
\int_{i_4}^{i_4+1} dx_4 \  [x_3-x_4]^{2H-2}_{\eta} [x_1-x_2]^{2H-2}_{\eta} \\
 \left([x_3-x_1]^{2H}_{\eta}+[i_3-i_1]^{2H}_{\eta}-[x_3-i_1]^{2H}_{\eta}
-[x_1-i_3]^{2H}_{\eta} \right) K_{i_2,i_4}(\eta;x_2,x_4),
\end{multline}
which is the sum of 4 terms, denoted in the sequel by $E_1,\ldots,E_4$. The most complicated one is a priori $E_1$,
obtained by choosing the contribution of $[x_3-x_1]^{2H}_{\eta}$ to
the integral. Let us first estimate this term. 

\smallskip

Apply  Lemma \ref{lem:exponents} with $f(x_4;u)=[u-x_4]_{\eta}^{2H-2}$,
$z=x_1$ ($x_4$ is simply an additional parameter here, and $f$ fulfills the analytic assumptions of Lemma \ref{lem:exponents} because $i_3$ and $i_4$ are not contiguous) and $\beta=2H,\gamma=0$:
letting
\begin{equation*}
\phi_1(x_4;x_1):=\int_{i_3}^{i_3+1} dx_3 [x_1-x_3]^{2H}_{\eta} [x_3-x_4]^{2H-2}_{\eta},
\end{equation*}
we obtain that $\phi_1$ is analytic in $x_1$ on a cut neighborhood $\Omega'_{cut}$ of 
$[i_1,i_1+1]$ excluding possibly $i_3$ and $i_3+1$ (depending on whether $i_3,i_3+1\in
\{i_1,i_1+1\}$ or not), and one can decompose $\phi_1$ into
\begin{equation}
\phi_1(x_4;x_1)=[x_1-i_3]^{2H+1}_{\eta} F_1(x_4;x_1)+G_1(x_4;x_1)
\label{eq:G1}
\end{equation}
on a neighborhood of $i_3$ (and similarly around $i_3+1$), with $F_1$ possibly
zero. The functions $\phi_1|_{\Omega'_{cut}}$, $F_1$ and $G_1$ are analytic
and bounded by a constant times $|i_3-i_4|^{2H-2}$.

\smallskip

Apply once again Lemma \ref{lem:exponents} with $f(x_4;u)=\phi_1(x_4;u)$,
$z=x_2$ and $\beta=2H-2$, $\gamma=0$ or (possibly) $2H+1$: letting
\begin{equation} \phi_2(x_4;x_2)=\int_{i_1}^{i_1+1} dx_1\ [x_2-x_1]^{2H-2}_{\eta} \phi_1(x_4;x_1), \end{equation}
$\phi_2$ is analytic in $x_2$ on a cut  neighborhood $\Omega''_{cut}$ of
$[i_2,i_2+1]$ excluding possibly $i_1$ and $i_1+1$, and
\begin{equation}
\phi_2(x_4;x_2)=[x_2-i_1]^{2H-1}_{\eta} F_2(x_4;x_2)+[x_2-i_1]^{4H}_{\eta} F_3(x_4;x_2)+G_2(x_4;x_2) 
\end{equation}
on a neighborhood of $i_1$ (and similarly around $i_1+1$), with the same
bounds as before for $\phi_2|_{\Omega''_{cut}}$, $F_2$, $F_3$ and $G_2$.

\smallskip

Finally, since $\phi_2$ is integrable with respect to $x_2$ on
$[i_2,i_2+1]$ and $K_{i_2,i_4}(\eta;x_2,x_4)$ is bounded by $C|i_3-i_4|^{2H-2}$
 by Lemma \ref{lem:K}, one gets
\begin{equation}
|E|\le C' \int_{i_4}^{i_4+1} dx_4\ |i_3-i_4|^{4H-4}=C'|i_3-i_4|^{4H-4}.
\end{equation}

\smallskip

There remain 3 'boundary' terms $E_2$, $E_3$, $E_4$ which are easier to cope
with. Consider for instance $E_3$ defined as
\begin{multline*}
E_3=\int_{i_4}^{i_4+1} dx_4  \int_{i_2}^{i_2+1} dx_2  \,
 K_{i_2,i_4}(\eta;x_2,x_4)  \\  \times
\int_{i_1}^{i_1+1} dx_1 \, [x_2-x_1]^{2H-2}_{\eta}
 \int_{i_3}^{i_3+1} dx_3\ [x_3-i_1]^{2H}_{\eta} [x_3-x_4]^{2H-2}_{\eta}. 
\end{multline*}
Applying again Lemma \ref{lem:exponents}, we get
\begin{multline*}
E_3=\\ C \int_{i_4}^{i_4+1} dx_4 G_1(x_4;i_1) 
\, \int_{i_2}^{i_2+1} dx_2
K_{i_2,i_4}(\eta;x_2,x_4) \left( [x_2-i_1-1]^{2H-1}_{\eta}-
[x_2-i_1]^{2H-1}_{\eta} \right),
\end{multline*}
where $G_1$ is as in eq. (\ref{eq:G1}). Since $x_2\mapsto [x_2-i_1-1]^{2H-1}_{\eta}$
and $x_2 \mapsto [x_2-i_1]^{2H-1}_{\eta}$ are integrable and
$G_1$, resp. $K_{i_2,i_4}$ is bounded by a constant times $|i_3-i_4|^{2H-2}$, one easily
gets an upper bound as the same form as before, namely, $|E_3|\le
C|i_3-i_4|^{4H-4}.$

\smallskip

We have thus proved that $E(i_1,\ldots,i_4)$ defined by (\ref{eq:44}) satisfies $E(i_1,\ldots,i_4)\le C |i_3-i_4|^{4H-4}$.
Finally, since $\sum\sum_{|i_3-i_4|\ge 2} |i_3-i_4|^{4H-4}=O(n)$ (as in eq. (\ref{eq:3:T2reg})), we obtain $\sum_{i_1,\ldots,i_4\in J_{3,1}}E(i_1,\ldots,i_4)=O(n)$.

\smallskip

Let us summarize now the results we have obtained so far: we have shown, respectively at Section \ref{sec:reg-terms} and \ref{sec:sing-terms}, that the terms $\ct_{j,{\rm reg}}$ and $\ct_{j,{\rm  sing}}$ defined by equation (\ref{eq:def-treg-tsing}) are $o(n^2)$. Going back to the definition of $\ct$ (see equation (\ref{eq:36})), this also shows that this quantity is of order $o(n^2)$. Recall now that $\be[ \tilde Z_n^4(\eta) ]_{(c)}$ can be decomposed into 6 terms, corresponding to our connected diagrams, each of the same kind as the particular example $\ct$ we have chosen. We have thus proved that $\be[ \tilde Z_n^4(\eta) ]_{(c)}=o(n^2)$ uniformly in $\eta$, which yields relation (\ref{eq:34}). This finishes the proof of Theorem \ref{thm:1.2} for $H<3/4$.

\section{Asymptotic error distribution of the Euler scheme: $H \geq 3/4$}\label{sec:Hgeq34}

  In this case, we derive the limit distribution in a different way, and first analyze the difference between the Euler and the Milstein scheme. An exact expression for this
difference is given by
\begin{align} \label{euler-milstein}
 \frac{1}{2} \sum_{i=0}^{n-1}  (B_{(i+1)/n}^{(1)}-B_{i/n}^{(1)}) (B_{(i+1)/n}^{(2)}-B_{i/n}^{(2)}), 
\end{align}
and we will see that, thanks to a simple geometric trick (borrowed from \cite{No}),
the latter quantity has the same law as
$$  \frac{1}{4} \sum_{i=1}^{n}  \left(  | B_{(i+1)/n}^{(1)}   - B_{i/n}^{(1)}|^2 -  | B_{(i+1)/n}^{(2)}   - B_{i/n}^{(2)}|^2 \right)
  ,$$

This allows to apply easily Theorem 2 in \cite{TV}, yielding the Lemma below, in which the following distribution appears:

\smallskip

\begin{defn}[Rosenblatt random variable]
A standard Rosenblatt random variable  with parameter $H_0=2H-1$  is given by
\begin{align*}
& \qquad \qquad  \frac{\left(  4H-3\right)  ^{1/2}}{4H\left(  2H-1\right) ^{1/2}}
\int_{0}^1 \int_0^1 \left(  \int_{ \max \{r, s \} }^{1}\frac{\partial K^{H}}{\partial u}\left(  u,s\right)  \frac{\partial K^{H}}{\partial
u}\left(  u,r\right)
du\right)  dW_r   dW_s 
\end{align*}
where $W$ is a standard Brownian motion,
$$K_{H}(t,s)=c_{H} s^{1/2-H} \int_{s}^{t} (u-s)^{H-3/2}u^{H-1/2} \, du \, 1_{[0,t)}(s)$$
and
$$c_{H} = \left( \frac{H(2H-1)}{\beta(2-2H,H-1/2)}\right)^{1/2}.$$
\end{defn}

\smallskip

\begin{lem}\label{tv_results}
The following limits in law hold true:

\smallskip

\noindent
{\it (i)}
Let $H=3/4$. Then we have
$$   \frac{\sqrt{2}n}{ \sqrt{ c_{1}(H)\log n}}    \sum_{i=0}^{n-1}  (B_{(i+1)/n}^{(1)}-B_{i/n}^{(1)}) (B_{(i+1)/n}^{(2)}-B_{i/n}^{(2)}) \stackrel{\mathcal{L}}{\longrightarrow} Z ,  $$
where $c_1(H)=9/16$ and  $Z$ is a standard normal random variable.

\smallskip

\noindent
{\it (ii)}
 Let $H \in (3/4,1) $. Then
$$   \frac{\sqrt{2} n}{ \sqrt{c_{2}(H)} }   \sum_{i=0}^{n-1}  (B_{(i+1)/n}^{(1)}-B_{i/n}^{(1)}) (B_{(i+1)/n}^{(2)}-B_{i/n}^{(2)}) \stackrel{\mathcal{L}}{\longrightarrow} \frac{1}{\sqrt{2}}(R_1-R_2) ,$$
where  $c_2(H)= 2H^{2}\left(  2H-1\right)  /\left(  4H-3\right) $ and  $R_1$ and $R_2$ are two independent standard Rosenblatt variables of index $2H-1$.
\end{lem}

\begin{proof}
{\it (i)}
Let $\beta$ be a fractional Brownian motion with Hurst index $H$ and define
\begin{align*}
V_{n}=\frac{1}{n}
\sum_{i=1}^{n}\left(  \frac{| \beta_{(i+1)/n}   -\beta_{i/n}|^2}{n^{-2H}}-1\right)  =  -1 + n^{2H-1}
\sum_{i=1}^{n}  | \beta_{(i+1)/n}   -\beta_{i/n}|^2.
\end{align*}
If $H=3/4$ it follows from \cite{TV} that
\begin{align} \label{tv=3/4_1}
 \sqrt{\frac{n}{c_{1}(H) \log (n)}}V_{n} \stackrel{\mathcal{L}}{\longrightarrow} Z, \end{align}
where $Z$ is  a standard normal random variable. Moreover, for $H \in (3/4,1)$ it is shown in  \cite{TV} that
\begin{align} \label{tv>3/4_1}
 \sqrt{\frac{n^{4-4H}}{c_{2}(H)}}V_{n} \stackrel{ \mathcal{L} }{\longrightarrow} R,
\end{align}
where $R$ is a standard Rosenblatt random variable with index $2H-1$.

Now let $\tilde{\beta}$ be another fractional Brownian motion with the same Hurst index as $\beta$, but independent of $\beta$
and define
$$V_{n}'=n^{2H-1} 
\sum_{i=1}^{n}  \left(  | \beta_{(i+1)/n}   -\beta_{i/n}|^2 -  | \tilde{\beta}_{(i+1)/n}   - \tilde{\beta}_{i/n}|^2
 \right).
$$
The continuous mapping theorem and  (\ref{tv=3/4_1}) implies that

\begin{align}   \label{tv=3/4_2}
 \sqrt{\frac{n}{c_{1}(H) \log (n)}}V_{n}' \stackrel{\mathcal{L}}{\longrightarrow}  Z_1 -Z_2
\end{align}
for $H=3/4$, where $Z_1$ and $Z_2$ are two independent  standard normal random variables. From (\ref{tv>3/4_1}) we obtain that
\begin{align}  \label{tv>3/4_2} 
\sqrt{\frac{n^{4-4H}}{c_{2}(H)}}V_{n}' \stackrel{\mathcal{L}}{\longrightarrow} (R_1-R_2),
\end{align}
where $R_1$ and $R_2$ are two independent standard Rosenblatt random variables with index $2H-1$.

\smallskip

\noindent
{\it (ii)}  Now, set $B^{(1)}=(\beta+\widetilde{\beta})/\sqrt{2}$ and $B^{(2)}=(\beta-\widetilde{\beta})/\sqrt{2}$. 
Then $B^{(1)}$ and $B^{(2)}$ are two independent fractional Brownian
motions with the same Hurst parameter. Moreover, we have
\begin{align*}
n^{2H-1}\sum_{k=0}^{ n -1} (B^{(1)}_{(k+1)/n}-B^{(1)}_{k/n}) 
(B^{(2)}_{(k+1)/n}-B^{(2)}_{k/n})
\stackrel{\mathcal{L}}{=} \frac{1}{2} V_n'.
\end{align*}
Thus, we have for $H=3/4$ that
\begin{align*}
\frac{2n}{ \sqrt{ c_{1}(H)\log n}}    \sum_{i=0}^{n-1}  (B_{(i+1)/n}^{(1)}-B_{i/n}^{(1)}) (B_{(i+1)/n}^{(2)}-B_{i/n}^{(2)}) 
\stackrel{\mathcal{L}}{=}  \sqrt{  \frac{n}{   c_{1}(H)\log (n) }}  V_n',
\end{align*}
and the first claim follows from  (\ref{tv=3/4_2}) and the fact that $Z_1-Z_2$ has the same distribution as $\sqrt{2}Z_1$.

Moreover,
since
$$    \frac{2n}{ \sqrt{c_{2}(H)}}   \sum_{i=0}^{n-1}  (B_{(i+1)/n}^{(1)}-B_{i/n}^{(1)}) (B_{(i+1)/n}^{(2)}-B_{i/n}^{(2)}) 
\stackrel{\mathcal{L}}{=}   \frac{n^{2-2H}}{ \sqrt{c_2(H)}}  V_n' 
$$
the second claim follows from  (\ref{tv>3/4_2}).

\end{proof}
 
Since the Milstein scheme has  a better convergence rate than the Euler scheme for $ H \geq 3/4$, the error of the latter scheme is dominated by (\ref{euler-milstein}). Thus, the asymptotic error distribution of the Euler scheme can be determined by the above Lemma, which will be carried out in the following two subsections.

\subsection{Error distribution of the Euler scheme for $\mathbf{H = 3/4}$}
By scaling we can assume  without loss of generality that $T=1$.
Recall that here we have
$$
\EX |X_1- X_1^n|^2= \frac{9}{128} \cdot  \log(n) n^{-2} + o(\log(n) n^{-2}  ).
$$
for the error of the Euler scheme.
Using the Milstein-type approximation $\widehat{X}_1^n$
we can write
\begin{align*}
 X_1- X_1^n &=      X_1-\widehat{X}_1^{n} + \widehat{X}_{1}^{n}-X_1^{n}  \\
& =  \frac{1}{2} \sum_{i=0}^{n-1}  (B_{(i+1)/n}^{(1)}-B_{i/n}^{(1)}) (B_{(i+1)/n}^{(2)}-B_{i/n}^{(2)}) + \rho_n,
\end{align*}
where $\rho_n= \widehat{X}_{1}^{n}-X_1^{n}$. Hence, setting $\ka_n:=n[\frac{9}{128} \log(n)]^{-1/2}$, we obtain
\begin{align*}
  \ka_n (X_1- X_1^n) 
& =  \frac{\ka_n}{2}    \sum_{i=0}^{n-1}  (B_{(i+1)/n}^{(1)}-B_{i/n}^{(1)}) (B_{(i+1)/n}^{(2)}-B_{i/n}^{(2)})  + \ka_n \rho_n.
\end{align*}
Now note that $  \ka_n \rho_n \rightarrow 0     $
in $L^2(\Omega)$ by Theorem \ref{milstein}
and
$$ \frac{\sqrt{2}n}{ \sqrt{ c_{1}(H)\log n}}    \sum_{i=0}^{n-1}  (B_{(i+1)/n}^{(1)}-B_{i/n}^{(1)}) (B_{(i+1)/n}^{(2)}-B_{i/n}^{(2)}) \stackrel{\mathcal{L}}{\longrightarrow} Z ,  $$
where $c_1(H)=9/16$ 
by Lemma \ref{tv_results}.
Since $[2/c_1(H)]^{1/2}=\frac12[128/9]^{1/2}$, it finally follows
that
$$      n (\log(n))^{-1/2}  ( X_1- X_1^n )  \stackrel{\mathcal{L}}{\longrightarrow}   \sqrt{ \frac{9}{128}} \cdot  Z,     $$
where $Z$ is a standard normal random variable.

\subsection{Error distribution of the Euler scheme for $\mathbf{H > 3/4}$}

Here we have 

$$ \EX|X_1 - X_1^n|^2 = \alpha_4(H) \cdot n^{-2} + o( n^{-2} )$$
with 
$$ \alpha_3(H) = \frac{1}{4} \frac{H^2(2H-1)}{4H-3}.$$
Proceeding as above, the limit distribution of the error of the Euler scheme is determined by
the limit distribution of
$$  \frac{n}{2\sqrt{\alpha_4(H)}} \sum_{i=0}^{n-1}  (B_{(i+1)/n}^{(1)}-B_{i/n}^{(1)}) (B_{(i+1)/n}^{(2)}-B_{i/n}^{(2)})  . $$
Since
$$  \frac{n}{2\sqrt{\alpha_4(H)}} =  \frac{4H-3}{H^2(2H-1)} =\frac{ \sqrt{2}}{\sqrt{c_3(H)}},$$ it follows by Lemma \ref{tv_results} that
 $$\frac{n}{2\sqrt{\alpha_4(H)}} (  X_1 - X_1^n   ) \stackrel{\mathcal{L}}{\longrightarrow} \frac{1}{\sqrt{2}} (R_1-R_2).$$



\end{document}